\documentclass[12pt]{amsart}
\usepackage{euscript,epsf,amssymb,xr}

\let\oldmarginpar\marginpar
\renewcommand\marginpar[1]
{\oldmarginpar{\tiny\bf \begin{flushleft} #1 \end{flushleft}}}

\input xy
\xyoption{all}

\newcommand{\la}{\langle}
\newcommand{\ra}{\rangle}

\newtheorem{theorem}{\bf Theorem}[section]
\newtheorem{lemma}[theorem]{\bf Lemma}

\newtheorem{prop}[theorem]{\bf Proposition}
\newtheorem{corollary}[theorem]{\bf Corollary}

\newcommand{\CC}{{\Bbb C}}

\newcommand{\CP}{{\Bbb CP}}

\newcommand{\FF}{{\Bbb F}}

\newcommand{\NN}{{\Bbb N}}

\newcommand{\RR}{{\Bbb R}}

\newcommand{\ZZ}{{\Bbb Z}}

\newcommand{\ggreat}{>\kern-.7ex>}
\newcommand{\ssmall}{<\kern-.7ex<}
\newcommand{\qu}{/\kern-.7ex/}
\newcommand{\exh}{\to\kern-1.8ex\to}

\newcommand{\cC}{{\EuScript{C}}}
\newcommand{\dD}{{\EuScript{D}}}

\newcommand{\fF}{{\EuScript{F}}}
\newcommand{\gG}{{\EuScript{G}}}

\newcommand{\jJ}{{\EuScript{J}}}

\newcommand{\lL}{{\EuScript{L}}}
\newcommand{\mM}{{\EuScript{M}}}
\newcommand{\oO}{{\EuScript{O}}}

\newcommand{\tT}{{\EuScript{T}}}
\newcommand{\xX}{{\EuScript{X}}}

\newcommand{\GL}{\operatorname{GL}}

\newcommand{\ab}{\operatorname{ab}}
\newcommand{\Aut}{\operatorname{Aut}}

\newcommand{\cyc}{\operatorname{cyc}}

\newcommand{\Diff}{\operatorname{Diff}}
\newcommand{\GCD}{\operatorname{GCD}}

\newcommand{\Hom}{\operatorname{Hom}}
\newcommand{\Id}{\operatorname{Id}}

\newcommand{\Int}{\operatorname{Int}}

\newcommand{\Ker}{\operatorname{Ker}}
\newcommand{\Lie}{\operatorname{Lie}}
\newcommand{\Mat}{\operatorname{Mat}}

\newcommand{\orb}{\operatorname{orb}}

\renewcommand{\O}{\operatorname{O}}
\newcommand{\Out}{\operatorname{Out}}

\newcommand{\SL}{\operatorname{SL}}

\newcommand{\PSL}{\operatorname{PSL}}

\newcommand{\SO}{\operatorname{SO}}

\newcommand{\Symp}{\operatorname{Symp}}

\renewcommand{\vert}{\operatorname{ver}}

\newcommand{\ov}{\overline}

\newcommand{\ord}{\operatorname{ord}}

\newcommand{\wt}{\widetilde}
\newcommand{\imag}{{\mathbf i}}

\title[Finite groups acting symplectically on $T^2\times S^2$]
{Finite groups acting symplectically on $T^2\times S^2$}

\author{Ignasi Mundet i Riera}
\address{Departament d'\`Algebra i Geometria\\
Facultat de Matem\`atiques\\
Universitat de Barcelona\\
Gran Via de les Corts Catalanes 585\\
08007 Barcelona \\
Spain}
\email{ignasi.mundet@ub.edu}

\date{\today}

\thanks{This work has been partially supported by the (Spanish) MEC Project MTM2012-38122-C03-02.}

\begin{document}

\maketitle

\begin{abstract}
For any symplectic form $\omega$ on $T^2\times S^2$ we construct
infinitely many nonisomorphic finite groups which admit
effective smooth actions on $T^2\times S^2$ that are trivial in
cohomology but which do not admit any effective symplectic
action on $(T^2\times S^2,\omega)$. We also prove that for any
$\omega$ there is another symplectic form $\omega'$ on
$T^2\times S^2$ and a finite group acting symplectically and
effectively on $(T^2\times S^2,\omega')$ which does not admit
any effective symplectic action on $(T^2\times S^2,\omega)$.

A basic ingredient in our arguments is the study of the Jordan
property of the symplectomorphism groups of $T^2\times S^2$. A
group $G$ is Jordan if there exists a constant $C$ such that
any finite subgroup $\Gamma$ of $G$ contains an abelian
subgroup whose index in $\Gamma$ is at most $C$. Csik\'os,
Pyber and Szab\'o proved recently that the diffeomorphism group
of $T^2\times S^2$ is not Jordan. We prove that, in contrast,
for any symplectic form $\omega$ on $T^2\times S^2$ the group
of symplectomorphisms $\Symp(T^2\times S^2,\omega)$ is Jordan.
We also give upper and lower bounds for the optimal value of
the constant $C$ in Jordan's property for $\Symp(T^2\times
S^2,\omega)$ depending on the cohomology class represented by
$\omega$. Our bounds are sharp for a large class of symplectic
forms on $T^2\times S^2$.
\end{abstract}

\section{Introduction}

\subsection{}
In this paper we study effective symplectic finite group
actions or, equivalently, finite subgroups of symplectomorphism
groups. Despite the extraordinary development of symplectic
geometry in the last three decades, the interactions between
finite transformation groups and symplectic geometry seems to
be so far a mostly unexplored terrain (with the remarkable
exceptions of \cite{C,CK1,CK2}).

The following notation will be useful in our discussion: for
any group $\gG$ we denote by $\fF(\gG)$ the set of all
isomorphism classes of finite subgroups of $\gG$.
Given a symplectic manifold $(X,\omega)$ we denote by $\Diff_{[\omega]}(X)$
the group of diffeomorphisms of $X$ which preserve the de Rham cohomology class represented
by $\omega$. We have inclusions
$$\fF(\Symp(X,\omega))\subseteq \fF(\Diff_{[\omega]}(X))\subseteq \fF(\Diff(X))$$
induced by the inclusions of the groups.
%
%
A basic question which apparently has not received attention is
the following: given a symplectic manifold $(X,\omega)$, how big can
the difference between $\fF(\Diff_{[\omega]}(X))$ and $\fF(\Symp(X,\omega))$ be?
Similarly, one may want to compare $\fF(\Symp(X,\omega))$ and $\fF(\Symp(X,\omega'))$ for
different symplectic structures $\omega,\omega'$.

If $\Sigma$ is a closed, connected and orientable surface, then
for any symplectic form $\omega$ on $\Sigma$ we have
$\fF(\Symp(\Sigma,\omega))=\fF(\Diff_{[\omega]}(\Sigma))=\fF(\Diff^+(\Sigma))$,
where $\Diff^+$ refers to orientation preserving
diffeomorphisms. To prove this claim, let us fix some
symplectic form $\omega$ on $\Sigma$.
Given a finite subgroup $\Gamma\subset\Diff^+(\Sigma)$ one may
take, by the averaging trick, a $\Gamma$-invariant Riemannian
metric $g$ on $\Sigma$; the volume form $\omega_g$ associated
to $g$ and the orientation given by $\omega$ is
$\Gamma$-invariant, and so is any constant multiple of
$\omega_g$. For some $\lambda\in\RR_{>0}$ we have an equality
of cohomology classes $[\lambda\omega_g]=[\omega]$ and by
Moser's stability there is a diffeomorphism
$\phi\in\Diff(\Sigma)$ such that
$\phi^*(\lambda\omega_g)=\omega$ (see e.g. Exercise 3.21 in
\cite{MS}). Conjugating the action of $\Gamma$ by $\phi$ we
obtain an action of $\Gamma$ which fixes $\omega$.

We will show in this paper that in higher dimensions the
situation becomes much more interesting. We will study in
detail $\fF(\Symp(T^2\times S^2,\omega))$ for every symplectic
form $\omega$ on $T^2\times S^2$, and we will prove that for
every $\omega$ the difference
$$\fF(\Diff_{[\omega]}(T^2\times
S^2))\setminus \fF(\Symp(T^2\times S^2,\omega))$$
contains infinitely many elements. Hence, there is an infinite
sequence of pairwise nonisomorphic finite groups
$G_1,G_2,\dots$ such that each $G_j$ acts smoothly and
effectively on $T^2\times S^2$ but, in contrast, there is no
effective symplectic action of $G_j$ on $(T^2\times
S^2,\omega)$. We will also prove that for any symplectic form
$\omega$ there exists another symplectic form $\omega'$ such
that
$$\fF(\Symp(T^2\times S^2,\omega'))\nsubseteq
\fF(\Symp(T^2\times S^2,\omega)),$$ i.e., there exists some
finite group $G$ which admits an effective symplectic action on
$(T^2\times S^2,\omega')$ but no such action on $(T^2\times
S^2,\omega)$.

A related question which we do not answer in this paper is
whether there exists some finite subgroup
$G\subset\Diff^+(T^2\times S^2)$ which does not admit effective
symplectic actions on $(T^2\times S^2,\omega)$ for any choice
of $\omega$ (this question is closely related to the results in
\cite{C,CK1,CK2}).

By a theorem of Lalonde and McDuff (see Theorem \ref{thm:LM}
below) the symplectic forms on $T^2\times S^2$ are classified
up to isomorphism by the ratio $\lambda$ between the volumes of
the $T^2$ factor and the $S^2$ factor. The theorems proved in
this paper imply that one can break the set $(0,\infty)$ of
possible values of $\lambda$ in infinitely many intervals of
the form $(a,b]$ so that if two choices of $\lambda$ belong to
different intervals then the corresponding symplectomorphism
groups have different families of isomorphism classes of finite
subgroups. From this perspective, our results are reminiscent
of those of Abreu and McDuff \cite{AM} on the rational homotopy
type of the symplectomorphism groups of $S^2\times S^2$.

Note that the theorem of Lalonde and McDuff implies that
$\fF(\Symp(T^2\times S^2,\omega))$ contains infinitely many
elements for every $\omega$.
In fact, for any $\omega$ and any $n$ there exists a subgroup
of $\Symp(T^2\times S^2,\omega)$ whose cardinal is $n$
(see the remarks after Theorem \ref{thm:LM}). Hence, any argument ruling
out the possibility that some finite group acts effectively and symplectically
on $\Symp(T^2\times S^2,\omega)$ must take into account more refined information
than the cardinal of the group.
The strategy we use in this paper to find obstructions for a finite
group $\Gamma$ to be isomorphic to a subgroup of
$\Symp(T^2\times S^2,\omega)$ is based on the notion of Jordan
group, which we next explain.

\subsection{Jordan groups}

A group $G$ is said to be Jordan \cite{Po0} if there is some
constant $C$ such that any finite subgroup $\Gamma$ of $G$
contains an abelian subgroup whose index in $\Gamma$ is at most
$C$. The terminology comes from a classic theorem of Camille
Jordan, which states that $\GL(n,\CC)$ is Jordan for every $n$
(see \cite{J} and \cite{B,CR} for modern presentations). A
number of papers have appeared in the last few years studying
whether the automorphism groups of different geometric
structures are Jordan or not: these include diffeomorphism
groups, groups of birational transformations of algebraic
varieties, or automorphism groups of algebraic varieties (see
\cite{Po2} for a survey).

Around twenty years ago, \'Etienne Ghys asked whether the
diffeomorphism group of any smooth compact manifold is Jordan
(see Question 13.1 in \cite{F}, and \cite{M2}).
This question has been answered affirmatively in a number of
cases (see the introduction and references in \cite{M2}).
For example, if $X$ is a smooth compact manifold with nonzero Euler
characteristic, then $\Diff(X)$ is Jordan (see \cite{M2} for a proof
in dimensions $2$ and $4$ and \cite{M3} for a proof in arbitrary dimensions
using the classification of finite simple groups).
However, Csik\'os, Pyber and Szab\'o \cite{CPS} proved recently
that the diffeomorphism group of $T^2\times S^2$ is not Jordan, thus
giving the first example of a compact manifold for which Ghys's question has
a negative answer
(see \cite{M4} for more examples). In contrast, in this paper
we prove that for any symplectic form $\omega$ on $T^2\times
S^2$ the group of symplectomorphisms $\Symp(T^2\times
S^2,\omega)$ is Jordan. Furthermore, we relate the constant in
Jordan property to the cohomology class represented by
$\omega$.

Consequently, from the perspective of Jordan property
$\fF(\Symp(T^2\times S^2,\omega))$ is {\it qualitatively}
smaller than $\fF(\Diff_{[\omega]}(T^2\times S^2))$ (the group actions
defined in \cite{CPS} are trivial in cohomology, so for any symplectic
form $\omega$ they give finite subgroups of $\Diff_{[\omega]}(T^2\times S^2)$).

To state our results with more precision we need to introduce some notation. Fix
orientations on $T^2$ and $S^2$ and choose elements $t\in T^2$
and $s\in S^2$. Define for any symplectic form $\omega$ on
$T^2\times S^2$
$$\alpha(\omega)=\int_{T^2\times\{s\}}\omega,\qquad
\beta(\omega)=\int_{\{t\}\times S^2}\omega.$$
The numbers $\alpha(\omega)$ and $\beta(\omega)$ are independent of $s$ and $t$ by Stokes' theorem.
Since $\omega$ is a symplectic form, both $\alpha(\omega)$ and $\beta(\omega)$
are nonzero. Define
$$\lambda(\omega)=\max \left\{\left(2\ZZ\cap \left(-\infty,\left|\frac{2\alpha(\omega)}{\beta(\omega)}\right|\right)\right)\cup\{1\}\right\}.$$
In words, $\lambda(\omega)$ is the biggest even integer smaller than $|2\alpha(\omega)/\beta(\omega)|$
if $|\alpha(\omega)/\beta(\omega)|>1$, and $\lambda(\omega)=1$ otherwise.

\begin{theorem}
\label{thm:main-1} Let $\omega$ be a symplectic form on
$T^2\times S^2$. Any finite subgroup $\Gamma\subset
\Symp(T^2\times S^2,\omega)$ contains an abelian subgroup
$A\subseteq\Gamma$ such that
$$[\Gamma:A]\leq \max\{144,6\lambda(\omega)\}.$$
\end{theorem}

The next theorem shows that
the bound in Theorem
\ref{thm:main-1} is optimal if $6\lambda(\omega)\geq 144$.

\begin{theorem}
\label{thm:main-2} Let $\omega$ be a symplectic form on
$T^2\times S^2$ such that $\lambda(\omega)\geq 8$. There exists
a finite subgroup $\Gamma\subset\Symp(T^2\times S^2,\omega)$
all of whose abelian subgroups $A\subseteq\Gamma$ satisfy
$[\Gamma:A]\geq 6\lambda(\omega)$. Furthermore, the action of
$\Gamma$ on the cohomology of $T^2\times S^2$ is trivial.
\end{theorem}

Combining Theorems \ref{thm:main-1} and \ref{thm:main-2} we immediately obtain:

\begin{corollary}
\label{cor:grups-que-no-actuen} For any symplectic form
$\omega$ on $T^2\times S^2$ the difference
$$\fF(\Diff_{[\omega]}(T^2\times S^2))\setminus
\fF(\Symp(T^2\times S^2,\omega))$$ contains infinitely many
elements; more precisely, there are infinitely many nonisomorphic finite
groups which admit smooth effective actions on $T^2\times S^2$
that are trivial in cohomology but which are not isomorphic to
any subgroup of $\Symp(T^2\times S^2,\omega)$.

Furthermore, for any symplectic form $\omega$ on $T^2\times
S^2$ there exists another symplectic form $\omega'$ such that
$$\fF(\Symp(T^2\times S^2,\omega'))\nsubseteq
\fF(\Symp(T^2\times S^2,\omega)).$$
\end{corollary}

If we restrict attention to finite $p$-groups for primes $p>3$ then our techniques give the following
sharp result.

\begin{theorem}
\label{thm:main-p}
Let $p>3$ be a prime and let $\omega$ be a symplectic form on $T^2\times S^2$.
The group $\Symp(T^2\times S^2,\omega)$ contains a nonabelian finite $p$-subgroup
if and only if $2p\leq\lambda(\omega)$. Furthermore, if $2p\leq\lambda(\omega)$ then there exists
a subgroup of $\Symp(T^2\times S^2,\omega)$ which is isomorphic to the Heisenberg $p$-group
$$\la X,Y,Z\mid X^p=Y^p=Z^p=[X,Z]=[Y,Z]=1,\,[X,Y]=Z\ra.$$
\end{theorem}

Combining Theorem \ref{thm:main-1} with the main result in \cite{M2} we obtain the following.

\begin{corollary}
\label{cor:main}
Let $(M,\omega)$ be a symplectic $4$-manifold diffeomorphic to the total
space of an $S^2$-fibration over a compact Riemann surface or to the product
of two compact Riemann surfaces. Then
$\Symp(M,\omega)$ is Jordan.
\end{corollary}

An important ingredient in the proofs of our theorems is a deep
result of Lalonde and McDuff \cite[Theorem 1.1]{LM} which has
been mentioned above and which classifies symplectic structures
on $T^2\times S^2$ (in fact the main theorem in \cite{LM}
applies to more general 4-manifolds, but we will only use the
result for $T^2\times S^2$). Fix symplectic forms
$\omega_{T^2}$ and $\omega_{S^2}$ on $T^2$ and $S^2$
respectively, both with total volume $1$.

\begin{theorem}[Lalonde, McDuff]
\label{thm:LM}
Let $\omega$ be a symplectic form on $T^2\times S^2$.
There exists a diffeomorphism $\phi$ of $T^2\times S^2$
such that $\phi^*\omega=\alpha(\omega)\omega_{T^2}+\beta(\omega)\omega_{S^2}$.
\end{theorem}

(Pullbacks are implicit in $\alpha(\omega)\omega_{T^2}+\beta(\omega)\omega_{S^2}$ and in similar
expressions appearing in the rest of the paper.)
An immediate consequence of Theorem \ref{thm:LM} is that for any symplectic form
$\omega$ on $T^2\times S^2$ there exist arbitrarily large finite nonabelian subgroups of
$\Symp(T^2\times S^2,\omega)$: by Moser's stability,
$\omega_{T^2}$ (resp. $\omega_{S^2}$) is isomorphic to the volume form associated to
a flat metric on $T^2$ (resp. a round metric on $S^2$); so we may take
for example a subgroup of $\Symp(T^2\times S^2,\omega)$ of the form $G_1\times G_2$,
where $G_1\subset\Symp(T^2,\omega_{T^2})$ is an arbitrary large finite abelian group
and $G_2\subset\Symp(S^2,\omega_{S^2})$ is
isomorphic to any finite nonabelian subgroup of $\SO(3,\RR)$.

By Theorem \ref{thm:LM},
to prove Theorem \ref{thm:main-1} it suffices to consider product symplectic forms
$\alpha\omega_{T^2}+\beta\omega_{S^2}$.
A standard technique in $4$-dimensional symplectic geometry,
based on pseudoholomorphic curves, allows us to prove that any symplectic finite group action
on $T^2\times S^2$ is equivalent to an action which preserves the fibration $T^2\times S^2\to T^2$
given by the projection to the first factor (Proposition \ref{prop:equivariant-fibration}).
The proof of Theorem \ref{thm:main-1} follows then from combining results
on finite group actions on $T^2$ and $S^2$ with a result on
finite group actions on line bundles over $T^2$ (Proposition \ref{prop:accions-fibrats-linia}).

To prove Theorem \ref{thm:main-2}
we observe that a slight modification of the construction in \cite{CPS} can be made symplectic.
(In particular, the groups in the statement of Theorem \ref{thm:main-2} can be taken to be
finite Heisenberg groups.)
This needs to be done carefully to estimate the cohomology class represented by the symplectic form.

Theorem \ref{thm:main-1} is proved in Section \ref{s:proof-thm:main}, Theorem \ref{thm:main-2}
is proved in Section \ref{s:proof-thm:main-2}, Theorem \ref{thm:main-p} is proved
in Section \ref{s:proof-thm:main-p},
and Corollary \ref{cor:main}
is proved in Section \ref{s:proof-cor:main}.

\subsection{Notation and conventions}
All manifolds and group actions in this paper will be implicitly assumed to be smooth.
As usual in the theory of finite transformation groups
in this paper $\ZZ_n$ denotes $\ZZ/n\ZZ$, not to be mistaken, when $n$ is a prime $p$,
with the $p$-adic integers. If $p$ is a prime we denote by $\FF_p$ the field of $p$ elements.
When we say that a group $G$ can be generated by $d$ elements we mean that there
are elements $g_1,\dots,g_d\in G$, {\it not necessarily distinct}, which generate $G$.
If a group $G$ acts on a set $X$ we denote
the stabiliser of $x\in X$ by $G_x$, and for any subset $S\subset G$ we denote
$X^S=\{x\in X\mid S\subseteq G_x\}$.

\subsection{Acknowledgements}
I would like to thank the referee for pointing out simplifications of several arguments
in the paper, and for his/her suggestions to enhance its readability.

\section{Proof of Theorem \ref{thm:main-1}}
\label{s:proof-thm:main}

\subsection{}
\label{ss:proof-main}
We prove Theorem \ref{thm:main-1}
modulo some results whose proofs are postponed
to later paragraphs of this section.
Denote throughout this section
$$X=T^2\times S^2$$
and let
$$\Pi:X\to T^2$$
be the projection to the first factor.
Take the product orientation on $T^2\times S^2$,
so that $\omega_{T^2}+\omega_{S^2}$ is compatible with the orientation.

Suppose that $\omega$ is a symplectic form on $X$ and that $\Gamma\subset\Symp(X,\omega)$
is a finite group.
Since both $S^2$ and $T^2$ admit orientation reversing diffeomorphisms we may assume,
replacing $\omega$ by $\theta^*\omega$ for a suitable diffeomorphism $\theta$ of $X$,
that
$$\alpha=\alpha(\omega)>0\qquad\text{and}\qquad \beta=\beta(\omega)>0.$$
(We then conjugate the original action of $\Gamma$ by $\theta$, so
that $\Gamma$ acts by symplectomorphisms with respect to $\theta^*\omega$.)
By Theorem \ref{thm:LM} there is a diffeomorphism $\xi$ of $X$ such that
$\xi^*\omega=\alpha \omega_{T^2}+\beta\omega_{S^2}$.
Conjugating the action of $\Gamma$ on $X$ by $\xi$ we may assume that
$$\Gamma\subset\Symp(X,\alpha \omega_{T^2}+\beta\omega_{S^2}).$$

Before continuing the proof, we introduce some useful terminology.
Suppose that $$q:E\to B$$ is a fibration of manifolds (by that we mean a locally
trivial fibration in the category of smooth manifolds, so in particular $q$
is a submersion). An action of a group $\Gamma$ on $E$
is said to be compatible with $q$ if it sends fibers of $q$
to fibers of $q$.
This implies that there is an action of $\Gamma$ on
$B$ such that if $x\in q^{-1}(b)$ then $\gamma\cdot x\in q^{-1}(\gamma\cdot b)$
for any $\gamma\in \Gamma$.

Let $\kappa_{S^2}\in H_2(X;\ZZ)$ be the homology class
represented by $\{t\}\times S^2$ for any $t\in T^2$, and let
$\kappa_{T^2}\in H_2(X;\ZZ)$ be the homology class
represented by $T^2\times \{s\}$ for any $s\in S^2$ (we use the
chosen orientations of $S^2$ and $T^2$). By Proposition
\ref{prop:equivariant-fibration}, there is an orientation
preserving diffeomorphism $\phi:X\to X$ such that the action of
$\Gamma$ on $X$ is compatible with the fibration
$\Pi\circ\phi$, and such that
$\phi_*\kappa_{S^2}=\kappa_{S^2}$, where $\phi_*$ is the map
induced in homology by $\phi$. Furthermore, there is a
$\Gamma$-invariant almost complex structure $J$ on $X$ which is
compatible with $\omega$ and with respect to which the fibers
of $\Pi\circ\phi$ are $J$-complex.

Since $\phi$ is orientation preserving, it preserves the
intersection pairing in $H_2(X;\RR)\simeq\RR^2$, which is hyperbolic.
We have $\phi_*\kappa_{S^2}=\kappa_{S^2}$ and $\kappa_{S^2}$ is
isotropic.
Hence the action of $\phi$ in $H_2(X;\RR)$ can be identified with
an element of $\O(1,1)$ fixing a nonzero isotropic vector. The
following lemma is an easy exercise in linear algebra.

\begin{lemma}
\label{lemma:O-1-1}
If $A\in\O(1,1)$ fixes a nonzero isotropic vector, then $A$ is the identity.
\end{lemma}

We deduce
that the action of $\phi$ on $H_2(X;\RR)$ is trivial. Hence $\phi$ acts
trivially on $H^2(X;\RR)$, so in particular $\phi^*[\omega]=[\omega]$.

Replacing $\omega$ by $\phi^*\omega$, and
conjugating both $J$ and the action of $\Gamma$ by $\phi$ we
put ourselves in the situation where
the action of
$\Gamma$ is compatible with $\Pi$ and $J$, and the fibers of $\Pi$ are
$J$-complex. The new symplectic form $\omega$ need no longer be a product
symplectic form, but it is compatible with the
almost complex structure $J$ and its cohomology class has not changed:
\begin{equation}
\label{eq:classe-omega}
[\omega]=\alpha[\omega_{T^2}]+\beta[\omega_{S^2}].
\end{equation}
Let $\Gamma_S\subseteq\Gamma$ be the subgroup whose
elements act trivially on the base of the fibration $\Pi$.
By Proposition \ref{prop:accions-en-fibracions}
at least one of the following sets of conditions holds true.
\begin{enumerate}
\item $\Gamma_S=\{1\}$.
\item There exists a nontrivial element $\gamma\in\Gamma_S$
such that $\Gamma$ preserves $X^{\gamma}$.
\item There exists a nontrivial element $\gamma\in\Gamma_S$
and a subgroup $\Gamma_0\subseteq \Gamma$ such that $[\Gamma:\Gamma_0]\leq 12$ and
$\Gamma_0$ preserves $X^{\gamma}$; furthermore, there is some
$h\in\Gamma_0\cap\Gamma_S$ such that for any $t\in T^2$ the action of $h$ on $\Pi^{-1}(t)$ exchanges
the two points of $\Pi^{-1}(t)\cap X^{\gamma}$.
\end{enumerate}

Suppose that $\Gamma_S=\{1\}$. Then the action of $\Gamma$ on $X$ comes from an
effective action of $\Gamma$ on $T^2$.
By Lemma \ref{lemma:tor} there is an abelian subgroup $A\subseteq\Gamma$
such that $[\Gamma:A]\leq 6$. So in this case the proof of the theorem is finished.

Suppose for the rest of the proof that we are in the second or third situation
given by Proposition \ref{prop:accions-en-fibracions}. To facilitate a unified
treatment, define $\Gamma_0:=\Gamma$ in case we are in the second situation.
Let $\gamma\in\Gamma_S$
be the nontrivial element referred to by the proposition. For any $t\in T^2$
the intersection $X^{\gamma}\cap\Pi^{-1}(t)$ consists of two points (see the comments before
Proposition \ref{prop:accions-en-fibracions}). By Lemma \ref{lemma:accions-fibracions}
the restriction of $\Pi$ to $X^{\gamma}$ is a fibration of manifolds. Hence,
$F:=X^{\gamma}$ is a two dimensional manifold and the restriction
$$p:\Pi|_F:F\to T^2$$
is a degree two covering map. Furthermore, $F$ is a $J$-complex submanifold of $X$.

By Proposition \ref{prop:superficies-dins-X},
$F$ is a compact
orientable surface which is either connected or has two
connected components, and the normal bundle $N\to F$ has a
structure of complex line bundle satisfying $\deg N=0$ if
$F$ is connected and $\deg N|_{F_1}+\deg N|_{F_2}=0$ if $F$ has
two connected components $F_1$ and $F_2$.
The degrees are
defined using an orientation on $F$ with respect to which the
projection $p$ is orientation preserving. Furthermore, by
Lemma \ref{lemma:accions-efectives}, the action of $\Gamma_0$
on the total space of $N$ is effective.

We treat separately the cases
$F$ connected and $F$ disconnected. In both cases we are going to
apply Proposition \ref{prop:accions-fibrats-linia} to the induced
action of $\Gamma_0$ to $N$ (or to its restriction $N|_{F_j}$).
This can be done because, as the action of $\Gamma$ preserves $J$
and $F$ is $J$-complex, the induced action of $\Gamma$ on $F$
is orientation preserving.

Suppose first of all that $F$ is connected. Then $\deg N=0$, so by
Proposition \ref{prop:accions-fibrats-linia} there is an abelian
subgroup $A\subseteq\Gamma_0$ satisfying $[\Gamma_0:A]\leq 6$. Since
in any case $[\Gamma:\Gamma_0]\leq 12$, we have $[\Gamma:A]\leq 72$,
so we are done.

Consider, for the rest of the proof, the case in which
$F$ has two connected components $F_1$ and $F_2$.

Suppose that there is some $h\in\Gamma_0\cap\Gamma_S$
such that for any $t\in T^2$ the action of $h$ on $\Pi^{-1}(t)$ exchanges
the two points of $\Pi^{-1}(t)\cap X^{\gamma}$. Then $h$ exchanges the
two connected components $F_1$ and $F_2$, and since the action of $h$
is compatible with $J$, we get an isomorphism of complex line bundles
$N|_{F_1}\simeq N|_{F_2}$. In view of the equality $\deg N|_{F_1}+\deg N|_{F_2}=0$
we obtain $\deg N|_{F_1}=\deg N|_{F_2}=0$.
Let $\Gamma_1\subseteq\Gamma_0$ be the subgroup preserving the connected components
$F_1,F_2$. By Lemma \ref{lemma:accions-efectives} the action of $\Gamma_1$
on $N|_{F_1}$ is effective. By Proposition \ref{prop:accions-fibrats-linia}
there is an abelian subgroup $A\subseteq\Gamma_1$ such that $[\Gamma_1:A]\leq 6$.
Combining all the estimates on indices we get
$$[\Gamma:A]=[\Gamma:\Gamma_0][\Gamma_0:\Gamma_1][\Gamma_1:A]
\leq 12\cdot 2\cdot 6=144,$$
so the proof is complete in this case.

Consider, to finish, the case in which no element of $\Gamma_0$ exchanges
the connected components $F_1,F_2$. In that case we have $\Gamma_0=\Gamma$.
We are going to bound the absolute value of the degrees
of $\deg N|_{F_j}$ in terms of the numbers $\alpha,\beta$.
Let $[F_j]\in H_2(X;\ZZ)$ be the homology class represented by
$F_j$ using the orientation on $F_j$ which is compatible with
$p$. Since $p$ restricts to a
diffeomorphism $F_j\to T^2$ for $j=1,2$, we have
$$[F_j]=\kappa_{T^2}+\lambda_j\kappa_{S^2}$$
for some integer $\lambda_j$.
Let $T^{\vert}=\Ker d\Pi\subset TX$ denote the vertical tangent
bundle of the fibration $\Pi$. We have $T^{\vert}=T^2\times
TS^2$, so $c_1(T^{\vert})=2[\omega_{S^2}]$ (the factor of $2$
is the Euler characteristic $\chi(S^2)$; recall that
$\omega_{S^2}$ has total volume $1$). Since $F$ intersects each fiber of
$\Pi$ transversely in two points, $N$ can be
identified with the restriction of $T^{\vert}$ to $F$, so we
have
$$\deg N|_{F_j}=\la c_1(T^{\vert}),[F_j]\ra=\la
2[\omega_{S^2}],\kappa_{T^2}+\lambda_j\kappa_{S^2}\ra=2\lambda_j.$$
Hence,
$$\lambda_j=\frac{\deg N|_{F_j}}{2}.$$
In particular, the degree $\deg N|_{F_j}$ is an even integer.
Since both $F_1$ and $F_2$ are $J$-complex
submanifolds and $J$ is compatible with $\omega$, we have,
using (\ref{eq:classe-omega}) and the fact that the total
volumes of $\omega_{T^2}$ and $\omega_{S^2}$ are $1$,
$$0<\la[\omega],[F_j]\ra=
\la\alpha[\omega_{T^2}]+\beta[\omega_{S^2}],\kappa_{T^2}+\lambda_j\kappa_{S^2}\ra
=
\alpha+\beta\lambda_j=\alpha+\beta\frac{\deg N|_{F_j}}{2}.$$
Consequently
$$\deg N|_{F_j}>-\frac{2\alpha}{\beta}$$
for $j=1,2$. Since $\deg N|_{F_1}+\deg N|_{F_2}=0$, this
implies that
$$|\deg N|_{F_j}|<\frac{2\alpha}{\beta},$$
and since $\deg N|_{F_j}$ is an even integer it follows that
$|\deg N|_{F_j}|\leq\lambda(\omega).$

By assumption $\Gamma_0$ preserves $F_1$, so by Lemma \ref{lemma:accions-efectives}
the action of $\Gamma_0$ on $N|_{F_1}$ is effective. By Proposition
\ref{prop:accions-fibrats-linia} there is
an abelian subgroup $A\subseteq\Gamma_0$ such that
$$[\Gamma_0:A]\leq 6\max\{1,|\deg N|_{F_1}|\}\leq 6\cdot \lambda(\omega).$$
Since $\Gamma_0=\Gamma$, the proof of Theorem \ref{thm:main-1} is complete.

\subsection{Construction of a $\Gamma$-invariant $S^2$-bundle structure}

Recall that $\kappa_{S^2}\in H_2(X;\ZZ)$ denotes the
homology class represented by $\{t\}\times S^2$ for any $t\in T^2$.

\begin{prop}
\label{prop:equivariant-fibration}
Let $\alpha,\beta$ be positive real numbers and consider the symplectic
form $\omega=\alpha \omega_{T^2}+\beta\omega_{S^2}$.
Suppose that a finite group $\Gamma$ acts symplectically on $(X,\omega)$.
There exists an orientation preserving diffeomorphism $\phi:X\to X$ such that
the action of $\Gamma$ is compatible with the fibration $\Pi\circ\phi$,
and a $\Gamma$-invariant almost complex structure $J$ on $X$
such that the fibers of $\Pi\circ\phi$ are $J$-complex. Finally we have
$\phi_*\kappa_{S^2}=\kappa_{S^2}$.
\end{prop}
\begin{proof}
The proof uses pseudoholomorphic curves and is a slight generalisation of
\cite[Proposition 4.1]{McD} and the note afterwards. We sketch the main
ideas for completeness, giving precise references when necessary (the reader
not familiar with pseudoholomorphic curve theory may look at the beautiful
survey \cite{LM2} for an introduction targeted to results on $4$-dimensional
ruled symplectic manifolds).
Let $\jJ$ denote the Fr\'echet space of $\cC^{\infty}$ almost
complex structures on $X$ which are compatible with
$\omega$.
The idea is that upon fixing any $J\in\jJ$ the $J$-holomorphic spheres cohomologous to $\kappa_{S^2}$
will fit into a fibration.

Fix a complex structure $J_{S^2}$
on $S^2$ compatible with the orientation. Let, for any $J\in\jJ$,
$$\mM(J)=\{u:S^2\to X\mid \ov{\partial}_Ju=0,\,u_*[S^2]=\kappa_{S^2}\}.$$
Here $\ov{\partial}_Ju=\frac{1}{2}(du\circ J_{S^2}-J\circ du)$ and
$[S^2]\in H_2(S^2;\ZZ)$ denotes the fundamental class
defined by the orientation. The group $G\simeq\PSL(2,\CC)$ of complex automorphisms of
$S^2$ acts on $\mM(J)$ by precomposition.
The compact open topology on the set of maps from $S^2$
to $X$ induces a topology on $\mM(J)$
with respect to which the action of $G$ is continuous and proper.
Gromov compactness theorem implies that $\mM(J)/G$ is compact
because one cannot write $\kappa_{S^2}=A_1+A_2$ in such a way that
both $A_1$ and $A_2$ belong to the image of the Hurewicz homomorphism
$\pi_2(X)\to H_2(X;\ZZ)$, and also $\la\omega,A_j\ra>0$ for $j=1,2$
(hence, no bubbling can occur).

Since $\la c_1(TX),\kappa_{S^2}\ra=2>1$, the main result in \cite{HLS}
(see also \cite[\S 3.3.2]{LM2}) implies that $\mM(J)$ has a
natural structure of smooth oriented manifold of dimension
$2(\la c_1(TX),\kappa_{S^2}\ra+1)=6$, and the action of $G$ on $\mM(J)$
is smooth. By the adjunction formula (see \cite[Exercise
3.5]{LM2}) each $u\in\mM(J)$ is an embedding. In particular,
the action of $G$ on $\mM(J)$ is free and $\mM(J)/G$ has a
natural structure of smooth oriented compact surface.

The natural evaluation map $\psi_J:\mM(J)\times_G S^2\to X$
that sends the class of $(u,s)\in\mM(J)\times S^2$ to $u(s)$
is an orientation preserving diffeomorphism (see \cite[Proposition 4.1]{McD} and the
note afterwards, and also \cite[\S 4.3]{LM2} --- the latter
refers only to fibrations over $S^2$, but everything works
identically for fibrations over general Riemann surfaces). The fact that
the evaluation map is orientation preserving is not explicitly mentioned neither
in \cite[Proposition 4.1]{McD} nor in \cite[\S 4.3]{LM2}, but it is an immediate consequence of
the fact that the evaluation map has degree $1$.
Using the multiplicativity of Euler characteristics in
fibrations, it follows that $\chi(\mM(J)/G)=0$, so that
$\mM(J)/G$ is diffeomorphic to $T^2$. Hence the projection
$f:\mM(J)\times_G S^2\to\mM(J)/G$ is a fibration over $T^2$
with fibers diffeomorphic to $S^2$, and its total space is
orientable.

It is well known that over a given surface there exist two oriented
$S^2$-fibrations up to isomorphism, the trivial one and a
twisted one (see e.g. \cite[Lemma 6.25]{MS}), and their total spaces are not
diffeomorphic. Therefore $\mM(J)\times_G S^2$ must be the trivial fibration
over $T^2$, so there exist
%
%
$$\xi\colon\mM(J)\times_G S^2\to X,\qquad
\eta\colon\mM(J)/G\to T^2$$
such that $\Pi\circ\xi=\eta\circ f$.

We emphasize that the preceding results hold true for {\it every} $J\in\jJ$.

Now let $\jJ_\Gamma\subset\jJ$ be the subset of
$\Gamma$-invariant almost complex structures (see
\cite[Proposition 5.49]{MS} and the comments before it).
For any $J\in\jJ_{\Gamma}$ the diffeomorphism
$$\phi:=\xi\circ\psi_J^{-1}:X\to X$$
and the almost complex structure $J$ satisfy the properties of the theorem.
Indeed, the fact that $\pi_2(T^2)=1$ implies that any diffeomorphism of $T^2\times S^2$
sends $\kappa_{S^2}$ to $\pm\kappa_{S^2}$. Since $\Gamma$ preserves $\alpha\omega_{T_2}+\beta\omega_{S^2}$, it follows that $\Gamma$ preserves $\kappa_{S^2}$.
Consequently $\Gamma$ acts on $\mM(J)$; this induces an action on
$\mM(J)\times_G S^2$ preserving the fibers of $\eta$ and with respect to which $\xi$ is $\Gamma$-equivariant.
\end{proof}


\subsection{Lemmas on finite groups acting on the sphere and the torus}

\begin{lemma}
\label{lemma:esfera-dos-punts} If $H$ is a nontrivial finite cyclic group acting
effectively and orientation preservingly on $S^2$ then $(S^2)^H$ consists
of two points.
\end{lemma}

Given two groups $H'\subseteq H$ we denote by $\Sigma_H(H')$ the collection of
all subgroups of $H$ which are equal to the image of $H'$ by some automorphism of $H$, i.e.
$$\Sigma_H(H')=\{\phi(H')\mid \phi\in\Aut(H)\}.$$
For example, $H'$ is a characteristic subgroup of $H$ if and only if $\Sigma_H(H')=\{H'\}$.

\begin{lemma}
\label{lemma:esfera}
Any nontrivial finite group $H$ acting effectively and orientation preservingly on $S^2$
has a nontrivial cyclic subgroup $H'\subseteq H$ such that at least one of these sets of
conditions is satisfied:
\begin{enumerate}
\item $|\Sigma_H(H')|\leq 1$,
\item $|\Sigma_H(H')|\leq 12$ and there is some $h\in H$ in the normalizer of $H'$
which exchanges the two points in $(S^2)^{H'}$.
\end{enumerate}
Furthermore, if $p>2$ is a prime and $H$ is a finite $p$-group acting effectively and
orientation preservingly on $S^2$ then $H$ is cyclic.
\end{lemma}

\begin{lemma}
\label{lemma:tor} Any finite group $H$ acting effectively and
orientation preservingly on $T^2$ has an abelian subgroup
$H'\subseteq H$ such that: $[H:H']\leq 6$, the action of $H'$
on $T^2$ is free, $H'$ is isomorphic to a subgroup of
$S^1\times S^1$, and the induced action of $H'$ on
$H^1(T^2;\ZZ)$ is trivial. Furhermore, if $p>3$ is a prime and
$H$ is a finite $p$-group acting effectively and orientation
preservingly on $T^2$, then the subgroup $H'$ can be chosen to
be $H$ itself.
\end{lemma}

To prove the preceding lemmas, we use the following argument.
If a finite group $H$ acts
by orientation preserving diffeomorphisms on a surface
$\Sigma$, then one may take an invariant Riemannian metric on
$\Sigma$ and consider the induced conformal structure. The
surface $\Sigma$ then becomes a Riemann surface, and the action
of $H$ on $\Sigma$ is by Riemann surface automorphisms. At this point we
may use results on automorphisms of Riemann surfaces to
understand the action of $H$.

\subsubsection{Proof of Lemmas \ref{lemma:esfera-dos-punts} and \ref{lemma:esfera}}
Lemma \ref{lemma:esfera-dos-punts} follows from Riemann's
uniformization theorem and the identification of the
automorphisms of $\Aut(\CP^1)$ with $\PSL(2,\CC)$.
For Lemma \ref{lemma:esfera} we use the classification
of the finite subgroups of $\PSL(2,\CC)$.
%
%
%
These coincide, up to conjugation, with those of $\SO(3,\RR)$,
because $\SO(3,\RR)\subset\PSL(2,\CC)$ is a maximal compact
subgroup.
Each finite subgroup of $\SO(3,\RR)$ is isomorphic to one of
these: a cyclic group $C_n$, a dihedral group $D_{2n}$ ($n\geq
3$), or the group $G_{12}$ (resp. $G_{24}$, $G_{60}$) of
orientation preserving isometries of a regular tetrahedron
(resp. cube, icosahedron), the subindex denoting the number of
elements (see e.g. \cite[Lect. 1]{Do}).

We prove Lemma \ref{lemma:esfera} treating separately each
case.
If $H\simeq C_n$ then we set $H':=H$, so $|\Sigma_H(H')|=1$. If
$H\simeq D_{2n}$ then we define $H'\subset H$ to be the
subgroup generated by all the elements of $H$ of order bigger
than $2$; the subgroup $H'$ is a nontrivial characteristic
cyclic subgroup of $H$, so $|\Sigma_H(H')|=1$. If $H\simeq
G_{12}$ then taking $H'\subset H$ to be any cyclic subgroup of
order $2$ we have $|\Sigma_H(H')|=3$; $H'$ can be identified
with the orientation preserving isometries of a regular
tetrahedron fixing the midpoints of two opposite edges, and
there is some orientation preserving isometry $h$ that
exchanges the two midpoints. If $H\simeq G_{24}$ then taking
$H'\subset H$ to be any cyclic subgroup of order $4$ we have
$|\Sigma_H(H')|=3$; $H'$ can be identified with the orientation
preserving isometries of a cube fixing the centers of two
opposite faces, and there is some orientation preserving
isometry $h$ that exchanges the centers of the two faces.
Finally, if $H\simeq G_{60}$ then taking $H'\subset H$ to be
any cyclic subgroup of order $5$ we have $|\Sigma_H(H')|=12$;
$H'$ can be identified with the orientation preserving
isometries of a regular icosahedron fixing two opposite
vertices, and there is some orientation preserving isometry $h$
that exchanges the two opposite vertices. The statement of
$p$-groups follows from the classification of finite subgroups
of $\SO(3,\RR)$.

\subsubsection{Proof of Lemma \ref{lemma:tor}}
We may identify $T^2$ with an elliptic curve $T=\CC/\Lambda$,
where $\Lambda\subset\RR^2\simeq\CC$ is a full rank lattice, in
such a way that $H$ acts on $T$ by complex automorphisms. Let
$\Aut_0(T)\subset \Aut(T)$ denote the subgroup of automorphisms
fixing the identity element $e$. We have
$\Aut(T)=T\cdot\Aut_0(T)$. It is well known that $\Aut_0(T)$
coincides with the group of discrete symmetries of the lattice
$\Lambda$ which are induced by complex linear automorphisms of
$\CC$, so $\Aut_0(T)$ is a cyclic group of order $2$, $3$,
$4$ or $6$. Hence $[\Aut(T):T]\leq 6$.
%
It follows that
$H':=H\cap T$ satisfies $[H:H']\leq 6$. Since $T$ is isomorphic
to $S^1\times S^1$ as a Lie group, $H'$ is isomorphic to an
abelian subgroup of $S^1\times S^1$. Since the action of $T$ on
itself is trivial in $H^1(T;\ZZ)$, so is the action of $H'$.
The statement on $p$-groups follows from the observation that
the only primes dividing an element of $\{2,3,4,6\}$ are $2$
and $3$.

\subsection{Lemmas on finite group actions and invariant submanifolds}

\begin{lemma}
\label{lemma:accions-efectives} Let $E$ be a compact and
connected manifold. Suppose that a finite group $H$ acts
effectively on $E$ and that $F\subset E$ is a $H$-invariant
submanifold. Let $N\to F$ be the normal bundle. The action of
$H$ on $E$ induces, linearising in the normal directions
of $F$, an effective action of $H$ on $N$ by bundle
automorphisms.
\end{lemma}

\begin{lemma}
\label{lemma:accions-fibracions} Let $q:E\to B$ be a
fibration of compact manifolds. Suppose that a finite group $H$
acts on $E$ compatibly with $q$, preserving an almost complex
structure $J$ on $E$, and preserving all fibers of $q$. Then for any subset
$U\subseteq H$ the
fixed point set $E^{U}$ is a $J$-complex submanifold and the
restriction of $q$ to $E^{U}$ is a fibration of manifolds.
\end{lemma}

The proofs of these lemmas are standard, so we just sketch the
main ideas. Suppose that a finite group $H$ acts on a compact
manifold $E$. Let $g$ be an $H$-invariant Riemannian metric on
$E$. Let $x\in E$ be any point, and let $H_x\subseteq H$ be its
isotropy group. The action of $H_x$ on $E$ induces a linear
action on $T_xE$, and the exponential map $\exp^g_x:T_xE\to E$
is $H_x$-equivariant. So, near $x$, $E^{H_x}$ is a submanifold
whose tangent space at $x$ can be identified with the linear
subspace $(T_xE)^{H_x}\subseteq T_xE$. Repeating the same
argument at each point of $E^{H_x}$ it follows that $E^{H_x}$
is a closed submanifold of $E$. The same argument implies that
$E^{H}$ is a closed submanifold of $E$.

If the action of $H$ on $E$ is effective and $E$ is connected
then for any nontrivial subgroup $H'\subseteq H$ the fixed
point set $E^{H'}$ has dimension smaller than that of $E$. This
implies that for any $x\in E^{H'}$ the linear action of $H'$ on
$T_xE$ identifies $H'$ with a subgroup of $\Aut(T_xE)$, and
hence is effective. Lemma \ref{lemma:accions-efectives} follows
immediately from this observation.

The proof of Lemma \ref{lemma:accions-fibracions} follows easily from the previous
arguments. Replacing $H$ by the subgroup generated by $U$ it suffices to consider
the case $U=H$.
For the last statement, note that by Ehresmann's theorem \cite{E} it suffices to check that
the restriction of $q$ to $E^{H}$ is a submersion.

\subsection{Finite groups of automorphisms of spherical fibrations over $T^2$}

Let $J$ be an almost complex structure on $X$ with respect to which
the fibers of $$\Pi:X\to T^2$$
are $J$-complex. The following observation is implicitly used in the next proposition.
If a finite group $G$ acts on $X$ preserving the fibers of $\Pi$ and respecting the
almost complex structure $J$ then for any nontrivial $g\in G$ and any $t\in T^2$ the
fixed point set $(\Pi^{-1}(t))^g$ consists of two points. This is a consequence of
Lemma \ref{lemma:esfera-dos-punts} and the fact that, since the action of $G$
preserves $J$ and the fibers of $\Pi$ are $J$-complex, the restriction of the action of
$G$ to any fiber of $\Pi$ is orientation preserving.

\begin{prop}
\label{prop:accions-en-fibracions}
Suppose that a finite group
$\Gamma$ acts effectively on $X$ respecting $J$,
and suppose that the action is compatible with the fibration
$\Pi$. Let $\Gamma_S\subseteq\Gamma$ be the subgroup whose
elements act trivially on the base of the fibration $\Pi$.
At least one of the following sets of conditions holds true.
\begin{enumerate}
\item $\Gamma_S=\{1\}$.
\item There exists a nontrivial element $\gamma\in\Gamma_S$
such that $\Gamma$ preserves $X^{\gamma}$.
\item There exists a nontrivial element $\gamma\in\Gamma_S$
and a subgroup $\Gamma_0\subseteq \Gamma$ such that $[\Gamma:\Gamma_0]\leq 12$ and
$\Gamma_0$ preserves $X^{\gamma}$; furthermore, there is some
$h\in\Gamma_0\cap\Gamma_S$ such that for any $t\in T^2$ the action of $h$ on $\Pi^{-1}(t)$ exchanges
the two points of $\Pi^{-1}(t)\cap X^{\gamma}$.
\end{enumerate}
\end{prop}

\begin{proof}
Let $\Gamma$ be a finite group acting effectively on $X$ and preserving both $J$ and $\Pi$.
As mentioned before, since the fibers of
$\Pi$ are $J$-complex, the induced action of $\Gamma$ on each
fiber of $\Pi$ is orientation preserving. Let
$\Gamma_S\subseteq\Gamma$ be the normal subgroup whose elements preserve the fibers of $\Pi$.
If $\Gamma_S=\{1\}$ then the proposition holds trivially.
So assume for the rest of the proof that $\Gamma_S\neq\{1\}$.

Let $S\subset X$ be any of the fibers of
$\Pi$. We claim that the action of $\Gamma_S$ on $S$
is effective. Indeed, if for some element
$\eta\in\Gamma_S$ we had $S^{\eta}=S$
then, since by Lemma \ref{lemma:accions-fibracions} the
projection $\Pi:X^{\eta}\to T^2$ is a fibration, we
would deduce that the fibers of $\Pi: X^{\eta}\to T^2$
are two dimensional closed submanifolds of the fibers of
$\Pi:X\to T^2$, hence $X^{\eta}=X$, contradicting the
assumption that $\Gamma$ acts effectively on $X$.

Since the action of
$\Gamma_S$ on $S$ is effective and orientation preserving, we may apply Lemma
\ref{lemma:esfera} and deduce that there is a nontrivial cyclic
subgroup $\Gamma_S'\subseteq \Gamma_S$ for which at least one of the following
two sets of conditions holds true.
\begin{enumerate}
\item $|\Sigma_{\Gamma_S}(\Gamma_S')|=1$;
\item $|\Sigma_{\Gamma_S}(\Gamma_S')|\leq 12$ and there is some $h\in\Gamma_S$
which normalizes $\Gamma_S'$ and which exchanges the two points in $S^{\Gamma_S'}$.
\end{enumerate}
In the first case we take $\gamma$ to be a generator of $\Gamma_S'$.
Then $X^{\gamma}=X^{\Gamma_S'}$ and, since $\Gamma_S'$ is a characteristic
subgroup of a normal subgroup $\Gamma_S$ of $\Gamma$, $\Gamma_S'$ is normal
in $\Gamma$. This implies that $X^{\Gamma_S'}$ (and hence also $X^{\gamma}$)
is preserved by $\Gamma$.

In the second case we take again generator $\gamma\in\Gamma_S'$
and we define
$$\Gamma_0=\{g\in\Gamma\mid g\Gamma_S'g^{-1}=\Gamma_S'\}.$$
Since $\Gamma_S$ is normal in $\Gamma$, $\Gamma_0$ satisfies
$[\Gamma:\Gamma_0]\leq |\Sigma_{\Gamma_S}(\Gamma_S')|\leq 12$.
Furthermore $\Gamma_0$ preserves $X^{\gamma}=X^{\Gamma_S'}$ because
$\Gamma_S'$ is normal in $\Gamma_0$. We claim that for any $t\in T^2$
the action of $h$ on $\Pi^{-1}(t)$ exchanges
the two points of $\Pi^{-1}(t)\cap X^{\gamma}$. Clearly $h\in\Gamma_0$,
because by assumption $h$ normalizes $\Gamma_S'$, so the action of
$h$ preserves $X^{\gamma}$. Since $h\in\Gamma_S$, the action of $h$ also preserves all the fibers
of $\Pi$. Applying Lemma \ref{lemma:accions-fibracions} to the
action to the subgroup $G\subseteq\Gamma_S$ generated by $h$ and the elements
of $\Gamma_S'$, it follows that the restriction of $\Pi$ to $X^G$ is a fibration
of manifolds. Since $X^G\cap S=\emptyset$, we deduce that $X^G=\emptyset$,
and this means that for any $t\in T^2$
the action of $h$ exchanges the two points in $\Pi^{-1}(t)\cap X^{\gamma}$.
\end{proof}

\begin{prop}
\label{prop:superficies-dins-X}
Suppose that $F\subset X$ is a $J$-complex closed submanifold intersecting transversely
each fiber of $\Pi$ and such that the restriction
of $\Pi$ to $F$ is a $2$-sheeted (unramified) covering $F\to T^2$.
Let $N\to F$ be the normal bundle of the inclusion $F\hookrightarrow X$,
endowed with the structure of complex line bundle inherited by $J$.
Then either $F$ is connected or it has two connected
    components $F_1$, $F_2$. In the first case, $F$ is
    diffeomorphic to $T^2$ and $\deg N=0$; in the
    second case, $F_j$ is diffeomorphic to $T^2$ for
    $j=1,2$ and $\deg N|_{F_1}+\deg N|_{F_2}=0$.
\end{prop}

The hypothesis of the proposition imply that $F$ is a compact orientable surface,
and to give a sense to the degree of $N$, we orient $F$ in
such a way that $p$ is orientation preserving.

\begin{proof}
Clearly, either
$F$ is connected or has two connected components. A computation
with the Euler characteristic shows that in the first case $F$
is a torus. In the second case the restriction of $p$ to each
connected component of $F$ is a diffeomorphism, so $F$ is the disjoint union of
two tori.

To prove the formulas on the degree of $N$, recall that on a real vector
bundle of rank two a choice of complex structure is equivalent up to homotopy
to a choice of orientation. Via this equivalence, the first Chern class is
equal to the Euler class. There is a natural (up to homotopy) isomorphism
between $N$ and the vertical tangent bundle of $\Pi$. Endowing the latter with
the orientation induced by $J$, this isomorphism is orientation preserving.
As a fibration of smooth oriented manifolds, we can identify $\Pi:X\to T^2$ with the
total space of $P\times_{\PSL(2,\CC)}\CP^1$, where $P$ is the trivial principal $\PSL(2,\CC)$-bundle.
But $P$ admits a reduction of the structure group to $\SO(3,\RR)$ with respect to
which $F$ is invariant under the antipodal map $X\to X$, because for any two distinct points $p,q\in\CP^1$ the
space $$\{f:\CP^1\to S^2\text{ conformal isomorphism }\mid f(p)=f(q)\text{ are antipodal}\}/\SO(3,\RR)$$
is contractible ($S^2$ is the round sphere in $\RR^3$). This implies the formulas on $\deg N$.
\end{proof}

\subsection{Finite groups of automorphisms of a complex line bundle over $T^2$}

\begin{prop}
\label{prop:accions-fibrats-linia} Let $L\to T^2$ be a complex
line bundle. Assume that a finite group $\Gamma$ acts
effectively on $L$ by vector bundle automorphisms and that
the induced action on $T^2$ is orientation preserving.
Then there is an
abelian subgroup $\Gamma_{\ab}\subseteq\Gamma$ satisfying
$$[\Gamma:\Gamma_{\ab}]\leq 6\cdot \max\{1,|\deg L|\}.$$
Suppose in addition that $\Gamma$ acts trivially on
$H^1(T^2;\ZZ)$ and that the induced action of $\Gamma$ on $T^2$
factors through a free action of an abelian quotient of
$\Gamma$ which can be generated by $2$ elements. Then there is
an abelian subgroup $\Gamma_{\ab}\subseteq\Gamma$ satisfying
$$[\Gamma:\Gamma_{\ab}]\leq \max\{1,|\deg L|\}.$$
\end{prop}
\begin{proof}
Let $\Gamma_0\subseteq\Gamma$ denote the subgroup consisting of
those elements which preserve the fibers of $L$. There is an
exact sequence $0\to \Gamma_0\to\Gamma\to\Gamma_B\to 0$, where
$\Gamma_B$ acts effectively and orientation preservingly on
$T^2$. By Lemma \ref{lemma:tor} there is an abelian subgroup
$\Gamma_B'\subseteq\Gamma_B$ such that
$[\Gamma_B:\Gamma_B']\leq 6$, $\Gamma_B'$ acts freely on $T^2$
and trivially on $H^1(T^2;\ZZ)$, and $\Gamma_B'$ can be
identified with a subgroup of $S^1\times S^1$. The latter
implies that $\Gamma_B'$ can be generated by two elements. So
if we replace $\Gamma$ by $\eta^{-1}(\Gamma_B')$, where
$\eta:\Gamma\to\Gamma_B$ is the quotient map, then we are in
the situation of the second statement. Consequently, the second
statement implies the first.

Let us prove the second statement. Assume that a finite group
$\Gamma$ acts effectively on a line bundle $L\to T^2$ and that
the induced action of $\Gamma$ on $T^2$ is orientation
preserving and factors through a free action of an abelian
quotient of $\Gamma$ which can be generated by $2$ elements. We
also assume that $\Gamma$ acts trivially on $H^1(T^2;\ZZ)$. If
$\Gamma$ is abelian then we set $\Gamma_{\ab}=\Gamma$ and we
are done. So we assume for the rest of the proof that $\Gamma$
is not abelian.

Let, as before, $\Gamma_0\subseteq
\Gamma$ denote the subgroup whose elements act trivially on the
base $T^2$, so that $\Gamma_B=\Gamma/\Gamma_0$ acts freely on
$T^2$, and $\Gamma_B$ is abelian and can be generated by
two elements. Let $\eta:\Gamma\to\Gamma_B$ be the quotient morphism.
We have an exact sequence of groups
$$1\to\Gamma_0\to\Gamma\stackrel{\eta}{\longrightarrow}\Gamma_B\to 1.$$
The subgroup $\Gamma_0\subset\Gamma$ is
central because its elements act by homothecies on the fibers of $L$
and the action of $\Gamma$ on $L$ is linear. Furthermore, the action
of $\Gamma$ on $L$ defines a monomorphism
$\Gamma_0\hookrightarrow S^1$, since the elements of $\Gamma_0$
act on $L$ as multiplication by a complex number of modulus
one. This implies that $\Gamma_0$ is cyclic.

Define a map
$$Q:\Gamma_B\times \Gamma_B\to\Gamma_0$$
as follows. Given elements $a,b\in\Gamma_B$ take lifts
$\alpha,\beta\in\Gamma$ and set
$$Q(a,b):=[\alpha,\beta]=\alpha\beta\alpha^{-1}\beta^{-1}.$$
The term $\alpha\beta\alpha^{-1}\beta^{-1}$
belongs to $\Gamma_0$ because $\Gamma_B$ is abelian, so $\eta(\alpha\beta\alpha^{-1}\beta^{-1})=1$.
It is straightforward to check that $[\alpha,\beta]$ only depends on $a$ and $b$, so $Q$ is well defined.

\begin{lemma}
\label{lemma:propietats-Q} The map $Q$ has the following
properties.
\begin{enumerate}
\item For all $a,b,c\in\Gamma_B$ we have
$Q(ab,c)=Q(a,c)Q(b,c)$,
$Q(a,bc)=Q(a,b)Q(a,c)$ and
$Q(a,a)=Q(1,a)=Q(a,1)=1$;
\item for any $a,b\in\Gamma_B$ the order of $Q(a,b)\in\Gamma$ divides
$\GCD(\ord_B(a),\ord_B(b))$, where $\ord_B$ refers to the order of elements in $\Gamma_B$;
\item if $p,q$ are different primes, $a\in\Gamma_B$ is a $p$-element and $b\in\Gamma_B$
is a $q$-element, then $Q(a,b)=1$;
\item if $a,b$ are both $p$-elements, the order of $Q(a,b)$ is at most $\max\{\ord_B(a),\ord_B(b)\}$.
\end{enumerate}
\end{lemma}
\begin{proof}
Suppose that $\alpha,\beta,\gamma\in\Gamma$ satisfy $\eta(\alpha)=a$, $\eta(\beta)=b$ and $\eta(\gamma)=c$. We have
\begin{align*}
Q(ab,c) &=(\alpha\beta)\gamma(\alpha\beta)^{-1}\gamma^{-1}
=\alpha\beta\gamma\beta^{-1}\alpha^{-1}\gamma^{-1}
=\alpha(\beta\gamma\beta^{-1}\gamma^{-1})\gamma\alpha^{-1}\gamma^{-1} \\
&=\alpha\gamma\alpha^{-1}\gamma^{-1}(\beta\gamma\beta^{-1}\gamma^{-1}) \qquad
\text{because $\beta\gamma\beta^{-1}\gamma^{-1}=[\beta,\gamma]$ is central} \\
&=Q(a,c)Q(b,c).
\end{align*}
The proof of $Q(a,bc)=Q(a,b)Q(a,c)$ is identical, and $Q(a,a)=Q(1,a)=Q(a,1)=1$ is immediate,
so (1) is proved. Using (1) we get
$Q(a,b)^{\ord_B(a)}=Q(a^{\ord_B(a)},b)=Q(1,b)=1$ and similarly
$Q(a,b)^{\ord_B(b)}=1$, which gives (2). Finally, (3) and (4) follow from (2).
\end{proof}

Let $\Gamma_c\subseteq\Gamma_0$ be the subgroup generated by
the elements $Q(a,b)\in\Gamma_0$ as $a,b$ run through all
elements of $\Gamma_B$. Clearly
$\Gamma_c=[\Gamma,\Gamma]$, so $\Gamma_c\neq \{1\}$ by assumption.

Before concluding the proof of Proposition \ref{prop:accions-fibrats-linia} we prove three lemmas.

Let $d_c=|\Gamma_c|$.

\begin{lemma}
\label{lemma:grau-divisible} $|\Gamma_B|$ divides the product
$d_c \deg L$.
\end{lemma}
\begin{proof}
Consider the line bundle $\Lambda=L^{\otimes d_c}$. The action of $\Gamma$ on
$L$ induces an action on $\Lambda$ defined by $\gamma\cdot (v_1\otimes\dots\otimes v_{d_c})=
(\gamma\cdot v_1)\otimes\dots\otimes (\gamma\cdot v_{d_c})$, and the subgroup of $\Gamma$ defined as $\Gamma_{\Lambda}^*=\{\gamma\in\Gamma\mid\gamma\text{ acts trivially on $\Lambda$}\}$
coincides with the set elements of $\Gamma_0$ whose order divides $d_c$. Since $\Gamma_0$ is
cyclic and $|\Gamma_c|=d_c$, we have $\Gamma_{\Lambda}^*=\Gamma_c$. The quotient
$\Gamma_{\Lambda}:=\Gamma/\Gamma_{\Lambda}^*=\Gamma/\Gamma_c=\Gamma/[\Gamma,\Gamma]$
acts effectively on $\Lambda$ and defining
$\Gamma_{\Lambda,0}:=\Gamma_0/\Gamma_c$ there is an exact sequence
$$1\to\Gamma_{\Lambda,0}\to\Gamma_{\Lambda}\to\Gamma_B\to 1.$$
The action of $\Gamma_{\Lambda}$ on $\Lambda$ gives a monomorphism
$i:\Gamma_{\Lambda,0}\hookrightarrow S^1$. Since $\Gamma_{\Lambda}$ is finite and abelian, there is a homomorphism $\rho:\Gamma_{\Lambda}\to S^1$ which extends $i$. Denote by
$$\phi:\Gamma_{\Lambda}\times\Lambda\to\Lambda$$
the map corresponding to the action of $\Gamma$ on $\Lambda$, so that $\phi(\gamma,\lambda)=\gamma\cdot\lambda$. Define a map
$$\psi:\Gamma_{\Lambda}\times\Lambda\to\Lambda$$
by $\psi(\gamma,\lambda)=\rho(\gamma)^{-1}\phi(\gamma,\lambda)$. The map $\psi$ defines a new action of $\Gamma$ on $\Lambda$, with respect to which $\Gamma_{\Lambda,0}$ acts trivially. Hence this new action factors through an action of $\Gamma_B$ on $\Lambda$ lifting the action on $T^2$. Since the action of $\Gamma_B$ on $T^2$ is free, so is the action of $\Gamma_B$ on $\Lambda$. Consequently, the bundle $\Lambda$
descends to a line bundle on the quotient $T^2/\Gamma_B$. Equivalently, there is a line bundle
$\Lambda'\to T^2/\Gamma_B$ satisfying $\Lambda\simeq q^*\Lambda'$, where $q:T^2\to T^2/\Gamma_B$
is the quotient map. Since $q$ has degree $|\Gamma_B|$, it follows that $\deg\Lambda$ is divisible by $|\Gamma_B|$. Finally, $\deg\Lambda=d_c\deg L$, so the proof is complete.
\end{proof}

\begin{lemma}
\label{lemma:grau-no-zero}
We have $\deg L\neq 0$.
\end{lemma}
\begin{proof}
Suppose that $\deg L=0$. We are going to prove that the action
of $\Gamma$ on $L$ factors through an abelian group. This is a
contradiction because by assumption $\Gamma$ is not abelian and
the action of $\Gamma$ on $L$ is effective.

Since $\deg L=0$, there is a nowhere vanishing smooth section
$\sigma:T^2\to L$. For any $\gamma\in\Gamma$ there is a unique
smooth map $\phi_{\gamma}:T^2\to\CC^*$ defined by the property
that
$\gamma\cdot\sigma(p)=\phi_{\gamma}(p)\cdot\sigma(\gamma\cdot
p)$ for every $p\in T^2$. For any $\gamma,\gamma'\in\Gamma$ the
equality
$\gamma'\cdot(\gamma\cdot\sigma(p))=(\gamma'\gamma)\cdot\sigma(p)$
gives the following cocycle condition
$$\phi_{\gamma'\gamma}(p)=\phi_{\gamma'}(\gamma\cdot
p)\phi_{\gamma}(p).$$ Denoting by $\rho_{\gamma}:T^2\to T^2$
the map $\rho_{\gamma}(p)=\gamma\cdot p$, we can rewrite the
cocycle condition as
$\phi_{\gamma'\gamma}=(\phi_{\gamma}'\circ\rho_{\gamma})\phi_{\gamma}.$
Associating to each map $T^2\to\CC^*$ its homotopy class and
using the canonical identification $[T^2,\CC^*]\simeq
H^1(T^2;\ZZ)$, each $\phi_{\gamma}$ corresponds to a cohomology
class $\Phi_{\gamma}\in H^1(T^2;\ZZ)$, and the cocycle
condition implies
$\Phi_{\gamma'\gamma}=\rho_{\gamma}^*\Phi_{\gamma'}+\Phi_{\gamma}$.
Since the action of $\Gamma$ on $H^1(T^2;\ZZ)$ is trivial, we
have $\rho_{\gamma}^*\Phi_{\gamma'}=\Phi_{\gamma'}$, so we have
$$\Phi_{\gamma'\gamma}=\Phi_{\gamma'}+\Phi_{\gamma}$$
for every $\gamma,\gamma'$. Now, $H^1(T^2;\ZZ)$ is torsion free
and $\Gamma$ is finite, so $\Phi_{\gamma}=0$ for every
$\gamma\in\Gamma$. So each $\gamma_{\gamma}$ is null homotopic,
and this implies that we can choose for every $\gamma$ a smooth
map $\psi_{\gamma}:T^2\to\CC$ such that
$\phi_{\gamma}=\exp(\psi_{\gamma})$.

Now let $\gamma,\gamma'\in\Gamma$ be arbitrary elements and let
$\zeta=[\gamma^{-1},\gamma'^{-1}]$, so that
$\gamma\gamma'=\gamma'\gamma\zeta$. We are going to prove that
$\zeta$ acts trivially on $L$. First note that, since the
action of $\Gamma$ on $T^2$ factors through an abelian
quotient, $\zeta$ acts trivially on $T^2$, so the cocycle
condition implies that
$$(\phi_{\gamma}\circ\rho_{\gamma'})\phi_{\gamma'}=\phi_{\gamma\gamma'}=\phi_{\gamma'\gamma\zeta}
=(\phi_{\gamma'\gamma}\circ\rho_{\zeta})\phi_{\zeta}=\phi_{\gamma'\gamma}\phi_{\zeta}
=(\phi_{\gamma'}\circ\rho_{\gamma})\phi_{\gamma}\phi_{\zeta}.$$
It follows that the smooth map $\chi:T^2\to\CC$ defined by the
equality
\begin{equation}
\label{eq:igualtat-chi}
\psi_{\gamma}\circ\rho_{\gamma'}+\psi_{\gamma}=\psi_{\gamma'}\circ\rho_{\gamma}+\psi_{\gamma}+\chi
\end{equation}
satisfies $\exp\chi=\phi_{\zeta}$ (note that $\chi$ need not be
equal to $\psi_{\zeta}$; what is true is that the difference
$\chi-\psi_{\zeta}$ is a constant integral multiple of
$2\pi\imag$). Let $\delta$ be the order of $\zeta$ in $\Gamma$.
Since $\zeta$ acts trivially on $T^2$ the cocycle condition for
$\phi_{\zeta}$ implies that $\phi_{\zeta}^{\delta}=1$. Hence
the condition $\exp\chi=\phi_{\zeta}$ implies that
$\chi(p)\in\delta^{-1}2\pi\imag\ZZ$ for every $p\in T^2$. Since
$\chi$ is smooth, we conclude that $\chi$ is constant. Fix any
point $p\in T^2$. It follows from (\ref{eq:igualtat-chi}) that
$$\sum_{\eta\in\Gamma}
\psi_{\gamma}(\gamma'\eta\cdot p)+\psi_{\gamma}(\eta\cdot p)=
\sum_{\eta\in\Gamma}\psi_{\gamma'}(\gamma\eta\cdot
p)+\psi_{\gamma}(\eta\cdot p)+\chi(\eta\cdot p).$$ Clearly
$\sum_{\eta\in\Gamma} \psi_{\gamma}(\gamma'\eta\cdot
p)=\sum_{\nu\in\Gamma} \psi_{\gamma}(\nu\cdot p)$ and
$\sum_{\eta\in\Gamma} \psi_{\gamma'}(\gamma\eta\cdot
p)=\sum_{\nu\in\Gamma} \psi_{\gamma'}(\nu\cdot p)$, so the
terms involving $\psi$'s in the equality above cancel each
other, and it follows that $\sum_{\eta\in\Gamma}\chi(\eta\cdot
p)=0$. Since $\chi$ is constant, this implies that $\chi=0$,
which implies $\phi_{\zeta}=1$, so $\zeta$ acts trivially on
$L$.
\end{proof}

\begin{lemma}
\label{lemma:cota-Gamma-B}
We have $d_c^2\leq |\Gamma_B|$.
\end{lemma}
\begin{proof}
We first prove that $\Gamma_c$ can be generated by an element of the form $Q(a,b)$ for some $a,b\in\Gamma_B$. Take to begin with a generator of $\Gamma_c$ of the form
$$h=Q(a_1,b_1)\cdot\dots\cdot Q(a_r,b_r).$$
Since $\Gamma_B$ is abelian we can write $a_i=\prod_p a_{ip}$, $b_i=\prod_p b_{ip}$,
where each product is over the set of primes, and $a_{ip}$, $b_{ip}$ are $p$-elements of $\Gamma_B$. In the next arguments
we use repeatedly Lemma \ref{lemma:propietats-Q}. We have
$$Q(a_i,b_i)=\prod_{p,q}Q(a_{ip},b_{iq})=\prod_p Q(a_{ip},b_{ip}),$$
and hence, if we denote by $\ord\gamma$ the order of any $\gamma\in\Gamma$,
$$\ord h=\ord \prod_i\prod_p Q(a_{ip},b_{ip})=\ord \prod_p\prod_i Q(a_{ip},b_{ip})\leq
\prod_p\max_i\ord Q(a_{ip},b_{ip}).$$
Choose for any $p$ an index $i(p)$ such that
$Q(a_{i(p)p},b_{i(p)p})=\max_i\ord Q(a_{ip},b_{ip})$. Let
$a=\prod_p a_{i(p)p}$ and $b=\prod_p b_{i(p)p}$. We have
$$d_c=\ord h\leq\prod_p\max_i\ord Q(a_{ip},b_{ip})=\ord Q(a,b).$$
This implies that $Q(a,b)$ is a generator of $\Gamma_c$. We claim that the set
$$S=\{a^ib^j\in\Gamma_B\mid 0\leq i<d_c,\,\,0\leq j<d_c\}$$
contains $d_c^2$ elements. Otherwise there would exist $0\leq k<d_c$ and $0\leq l<d_c$ such that $a^kb^l=1$, hence $b^{-l}=a^k$. This would
imply $Q(a,b)^k=Q(a^k,b)=Q(b^{-l},b)=Q(b,b)^{-l}=1$. Hence $\ord Q(a,b)<d_c$, a contradiction with our previous computation. It follows that $\Gamma_B$ contains at least $d_c^2$ elements, so the lemma is proved.
\end{proof}

We are now ready to finish the proof of Proposition \ref{prop:accions-fibrats-linia}.
By Lemma \ref{lemma:grau-no-zero}
we have $\deg L\neq 0$. By Lemma \ref{lemma:grau-divisible}, the nonvanishing of $\deg L$ implies that $|\Gamma_B|\leq |d_c\deg L|$. Using this inequality and Lemma \ref{lemma:cota-Gamma-B} we have
$$|\Gamma_B|^2\leq d_c^2(\deg L)^2\leq |\Gamma_B|(\deg L)^2.$$
Dividing both sides by $|\Gamma_B|$ we get
$$|\Gamma_B|\leq (\deg L)^2.$$
Since $\Gamma_B$ can be generated by two elements,
there are three possibilities:
$\Gamma_B$ is trivial, $\Gamma_B$ is nontrivial cyclic,
or $\Gamma_B$ is isomorphic to $\ZZ_{n_1}\times\ZZ_{n_2}$
where $n_1,n_2$ are natural numbers bigger than one. In each of the three cases
there exists a cyclic subgroup $\Gamma_{\cyc}\subseteq\Gamma_B$ such that
$[\Gamma_B:\Gamma_{\cyc}]\leq |\Gamma_B|^{1/2}\leq |\deg L|$. Define
$$\Gamma_{\ab}:=\eta^{-1}(\Gamma_{\cyc}).$$
By (1) in Lemma \ref{lemma:propietats-Q}, $\Gamma_{\ab}$ is abelian. Finally,
$[\Gamma:\Gamma_{\ab}]\leq |\deg L|$, so we are done.
\end{proof}

\section{Proof of Theorem \ref{thm:main-2}}
\label{s:proof-thm:main-2}

The first three subsections of this section are devoted to introducing the preliminaries
of the proof of Theorem \ref{thm:main-2}, which is given in Subsection \ref{ss:proof-thm:main-2}.

\subsection{The group $\Gamma_n$}
Let $I$ be an ideal of a
commutative ring $R$ with unit. Consider the group
$$T(R,I)=\left\{A(x,y,z):=\left(\begin{array}{ccc}
1 & x & z \\
0 & 1 & y \\
0 & 0 & 1 \end{array}\right)\in\Mat_{3\times 3}(R)\mid x,y,z\in
I\right\}$$ with the group structure given by matrix
multiplication. For any natural number $n$,
$T(\ZZ,n\ZZ)$ is a normal subgroup of $T(\ZZ,\ZZ)$, so we may
define the quotient group
$$\Gamma_n:=T(\ZZ,\ZZ)/T(\ZZ,n\ZZ).$$
The map
$$\eta:\Gamma_n\to V:=\ZZ_n\times\ZZ_n$$
which sends the class of $A(x,y,z)$ to $([x],[y])$
is a surjective morphism of groups. The kernel of $\eta$ can be identified with
$\Gamma_n^0=\{[A(0,0,z)]\mid z\in\ZZ\},$
which is the center of $\Gamma_n$.
The map $\psi:\Gamma_n^0\to \ZZ_n$ that sends $[A(0,0,z)]$ to $[z]$ is an isomorphism of groups.
Hence $\Gamma_n$ sits in an exact sequence of groups
$$0\to\ZZ_n\to\Gamma_n\stackrel{\eta}{\longrightarrow}\ZZ_n\times\ZZ_n\to 0.$$
The group $\Gamma_n$ is sometimes called a finite Heisenberg group.
When $n$ is a prime $p$, $\Gamma_n$ is isomorphic to the group in the statement
of Theorem \ref{thm:main-p}.

\begin{lemma}
\label{lemma:abelian-subgroups-Gamma-n}
For any abelian subgroup $A\subseteq\Gamma_n$ we have $[\Gamma_n:A]\geq n$.
\end{lemma}
\begin{proof}
This is proved in Section 3 of \cite{Z} (note that
$\Gamma_n\simeq{\mathfrak G}_K^1$ taking $N=n$ in \cite{Z}).
\end{proof}

\subsection{The circle bundle $M_n\to T_n^2$} Fix a natural number $n$.
Let
$$T_n^2:=\RR^2/n\ZZ^2$$ with its natural smooth structure. The group
$\ZZ_n\times\ZZ_n$ acts on $T_n^2$ in the obvious way: $([a],[b])\cdot [(x,y)]=[(a+x,b+y)]$.

Define
$$M_n:=T(\ZZ,n\ZZ)\backslash T(\RR,\RR).$$
Endow $T(\RR,\RR)$ with the structure of differential manifold with respect to which
$\RR^3\ni (x,y,z)\mapsto A(x,y,z)\in T(\RR,\RR)$ is a diffeomorphism.
Since the action of $T(\ZZ,n\ZZ)$ on $T(\RR,\RR)$ is smooth and properly discontinuous,
$M_n$ has a natural structure of differential manifold.
The group $\Gamma_n$ acts smoothly and effectively on $M_n$ on the left via
product of matrices. On the other hand, the projection
$T(\RR,\RR)\ni A(x,y,z)\mapsto (x,y)\in\RR^2$ descends to a
projection
$$\rho:M_n\to T^2_n$$
which is a principal circle bundle. The structure of principal
bundle is induced by right multiplication on $T(\RR,\RR)$ by
central elements. More concretely,
\begin{equation}
\label{eq:accio-cercle-Mn}
e^{2\pi\imag t}\cdot[A(x,y,z)]=[A(x,y,z)A(0,0,nt)].
\end{equation}

The action of
$\Gamma_n$ on $M_n$ is by principal bundle automorphisms,
lifting the action of $\Gamma_n$ on $T^2_n$ defined through
the map $\eta:\Gamma_n\to \ZZ_n\times\ZZ_n$ and the action of $\ZZ_n\times\ZZ_n$
on $T_n^2$ defined above.

We identify the tangent space $T_{\Id}T(\RR,\RR)$ with the set of
$3\times 3$ upper diagonal real matrices with zeroes in the diagonal,
namely
\begin{equation}
\label{eq:tangent-T(R,R)}
T_{\Id}(\RR,\RR)=\{\alpha(x,y,z)=A(x,y,z)-A(0,0,0)\mid x,y,z\in\RR\}.
\end{equation}
Let
$$e_x=(1,0,0),\qquad e_y=(\cos 2\pi/6,\sin2\pi/6,0),\qquad e_z=(0,0,1)$$
and consider the isomorphism of vector spaces
$$f:T_{\Id}(\RR,\RR)\to\RR^3,\qquad f(\alpha(x,y,z))=xe_x+ye_y+ze_z.$$
Consider the left invariant Riemannian metric $\wt{g}$ on $T(\RR,\RR)$ whose restriction
to $T_{\Id}T(\RR,\RR)$ is the pairing
$$\la\alpha,\alpha'\ra:=\la f(\alpha),f(\alpha')\ra_{\RR^3},$$
where $\la\cdot,\cdot\ra_{\RR^3}$ denotes the Euclidean pairing
in $\RR^3$. We use this choice of metric because the $\ZZ$-span
of the vectors $e_x,e_y$ is a lattice in the plane
$\{(a,b,c)\mid c=0\}$ with rotational $\ZZ_6$-symmetry; this
will be crucial in Subsection \ref{ss:Z6-symmetry}.

By invariance, the metric $\wt{g}$ descends to a metric $g_n$ on $M_n$. The metric $g_n$
on $M_n$ is also $S^1$-invariant, since the action 
of $S^1$ on $M_n$ is defined via multiplication by
central elements of $T(\RR,\RR)$, i.e.
$A(x,y,z)A(0,0,nt)=A(0,0,nt)A(x,y,z)$.

\subsection{Introducing an extra $\ZZ_6$-symmetry}
\label{ss:Z6-symmetry} Define the following smooth map
$$h:T(\RR,\RR)\to T(\RR,\RR),\qquad
h(A(x,y,z))=A\left(-y,x+y,z-xy-\frac{1}{2}y^2\right).$$
A simple but tedious computation proves that $h^6=\Id$ (so in particular
$h$ is a diffeomorphism) and that
$h$ is an morphism (hence an isomorphism) of groups:
$$h(A(x,y,z))h(A(x',y',z'))=h(A(x,y,z)A(x',y',z')).$$
The definition of $h$ may seem a bit awkward, specially for the presence
of a quadratic terms. See the Appendix for a geometric interpretation of $h$
in which these quadratic terms come up from an easy computation with iterated integrals.

The identity element $A(0,0,0)$ is fixed by $h$ and the action
on $T_{\Id}T(\RR,\RR)$ induced by $h$ is the linear map which,
in terms of (\ref{eq:tangent-T(R,R)}), takes the form
$$\alpha(x,y,z)\mapsto \alpha(-y,x+y,z).$$
It follows that $h$ fixes the Riemannian metric $\wt{g}$ defined in the previous subsection.

Suppose for the rest of this subsection that $n$ is an {\it
even} natural number. Then $h$ preserves $T(\ZZ,n\ZZ)$, so $h$
gives rise to a diffeomorphism $h_n$ of $M_n$ which is a
$g_n$-isometry. Furthermore, since $h$ acts trivially on the
subgroup $\{A(0,0,z)\mid z\in\RR\}\subset T(\RR,\RR)$, the
action of $h_n$ commutes with the $S^1$-action on $M_n$, so
$h_n$ acts by principal bundle automorphisms on $M_n\to T_n^2$.

\newcommand{\wh}{\widehat}

Let $\wh{\Gamma}_n\subset\Diff(M_n)$ be the subgroup generated by (the action on $M_n$
of the elements of) $\Gamma_n$ and $h_n$. Combining our previous observations on the
action of $\Gamma_n$ and $h$, we deduce that $\wh{\Gamma}_n$ acts on $M_n$ by
$S^1$-principal bundle automorphisms and by $g_n$-isometries.

\begin{lemma}
\label{lemma:abelian-subgroups-wh-Gamma}
If $n\geq 8$ then any abelian subgroup $A\subseteq\wh{\Gamma}_n$ satisfies
$[\wh{\Gamma}_n:A]\geq 6n$.
\end{lemma}
\begin{proof}
Let $B_n\subset\Diff(T_n^2)$ be the subgroup generated by the diffeomorphisms
$\chi,t_a,t_b\in\Diff(T_n^2)$
defined as
$$\chi([x],[y])=([-y],[x+y]),\quad
t_a([x],[y])=([x+1],[y]),\quad
t_b([x],[y])=([x],[y+1]).$$
Since $\chi^{-1}t_a\chi=t_at_b^{-1}$ and
$\chi^{-1}t_b\chi=t_a$ (we omit the symbol $\circ$ in the compositions)
the subgroup $\la t_a,t_b\ra$, which is
isomorphic to $\ZZ_n\times\ZZ_n$, is a normal subgroup of $B_n$. Hence, there
is an exact sequence
$$0\to\ZZ_n\times\ZZ_n\to B_n\stackrel{\zeta}{\longrightarrow}\ZZ_6\to 0,$$
where $\zeta(\chi)\in\ZZ_6$ is a generator and the element $(u,v)\in\ZZ_n\times\ZZ_n$
is mapped to $t_a^ut_b^v$. Furthermore,
the action of $\ZZ_6$ on $\ZZ_n\times\ZZ_n$ given by conjugation
in $B_n$ is $\chi\cdot(u,v)=(u+v,-u)$.

Suppose that $A\subseteq B_n$ is an abelian subgroup and that
$\zeta(A)\neq 0$. There are three possibilities for the image $\zeta(A)$. Suppose first that
$\zeta(A)=\ZZ_6$. Then for any $(u,v)\in A\cap\Ker\zeta\subseteq\ZZ_n\times\ZZ_n$ we have
$\chi\cdot (u,v)=(u+v,-u)=(u,v)$, which implies $(u,v)=(0,0)$, i.e.,
$$\zeta(A)=\la\chi\ra\quad\Longrightarrow\quad A\cap\Ker\zeta=0.$$
Next suppose that $\zeta(A)=\la\chi^2\ra\subset\ZZ_6$. Then for any $(u,v)\in A\cap\Ker\zeta\subseteq\ZZ_n\times\ZZ_n$ we have
$\chi^2\cdot (u,v)=(v,-u-v)=(u,v)$, which implies $(u,v)=(0,0)$ if $n$ is not divisible by $3$
and $(u,v)\in \{(0,0),(n/3,n/3)\}$ if $n$ is divisible by $3$; in any case,
$$\zeta(A)=\la\chi^2\ra\quad\Longrightarrow\quad A\cap\Ker\zeta\subseteq K_2:=\{(0,0),(n/3,n/3)\},$$
where we agree that the second term only appears if
$n$ is divisible by $3$. Finally, suppose that $\zeta(A)=\la\chi^3\ra$.
Then for any $(u,v)\in A\cap\Ker\zeta\subseteq\ZZ_n\times\ZZ_n$ we have
$\chi^3\cdot (u,v)=(-u,-v)=(u,v)$, which implies $(u,v)\in\{0,n/2\}\times\{0,n/2\}$;
hence
$$\zeta(A)=\la\chi^3\ra\quad\Longrightarrow\quad A\cap\Ker\zeta\subseteq K_3:=\{0,n/2\}\times\{0,n/2\}.$$

It is immediate from the definitions
that there is a morphism of groups $\wh{\eta}:\wh{\Gamma}\to B_n$ with the property
that each $\phi\in\wh{\Gamma}_n$, seen as a diffeomorphism of $M_n$, lifts
$\wh{\eta}(\phi)$. Setting
$\theta=\zeta\circ\wh{\eta}$ we have a commutative diagram
$$\xymatrix{0 \ar[r] & \Gamma_n \ar[d]^{\eta}\ar[r] & \wh{\Gamma}_n\ar[d]^{\wh{\eta}}\ar[r]^{\theta}
& \ZZ_6 \ar[r]\ar@{=}[d] & 0 \\
0 \ar[r] & \ZZ_n^2\ar[r] & B_n\ar[r]^{\zeta} & \ZZ_6\ar[r] & 0.}$$
Suppose that $\wh{A}\subseteq\wh{\Gamma}_n$ is abelian. Then
$\wh{\eta}(\wh{A})\subseteq B_n$ is also abelian.
We are going to bound $[\wh{\Gamma}_n:\wh{A}]$
treating different cases separately.
If $\zeta(\wh{\eta}(\wh{A}))=0$. then
$\wh{A}\subseteq\Ker\theta$, so $\wh{A}$ can be identified with an abelian subgroup
of $\Gamma_n$. By Lemma \ref{lemma:abelian-subgroups-Gamma-n} we have
$$[\wh{\Gamma}_n:\wh{A}]=6[\Gamma_n:\wh{A}]\geq 6n.$$
If $\zeta(\wh{\eta}(\wh{A}))=\ZZ_6$ then, by our previous comment,
$\wh{\eta}(\wh{A})\cap\Ker\zeta=0$, which implies
that $\wh{A}\cap\Ker\theta\subseteq\Ker\eta$. This implies that
$|\wh{A}|\leq 6|\Ker\eta|=6n$, so
$$[\wh{\Gamma}_n:\wh{A}]\geq \frac{6n^3}{6n}=n^2\geq 6n.$$
If $\zeta(\wh{\eta}(\wh{A}))=\la\chi^2\ra$ then
$\wh{A}\cap\Ker\theta\subseteq\eta^{-1}(K_2)$, so
$|\wh{A}|\leq |\la\chi^2\ra|\cdot |\eta^{-1}(K_2)|=6\cdot |\Ker\eta|=6n,$
which gives
$$[\wh{\Gamma}_n:\wh{A}]\geq \frac{6n^3}{6n}=n^2\geq 6n.$$
If $\zeta(\wh{\eta}(\wh{A}))=\la\chi^3\ra$ then
$\wh{A}\cap\Ker\theta\subseteq\eta^{-1}(K_3)$, so
$|\wh{A}|\leq |\la\chi^3\ra|\cdot |\eta^{-1}(K_3)|=8\cdot |\Ker\eta|=8n,$
which gives
$$[\wh{\Gamma}_n:\wh{A}]\geq \frac{6n^3}{8n}=\frac{6n^2}{8}\geq 6n,$$
so the proof of the lemma is complete.
\end{proof}

\subsection{A $\widehat{\Gamma}_n$-invariant symplectic form on $M_n\times_{S^1}S^2$}

\newcommand{\FS}{\operatorname{FS}}
\newcommand{\oFS}{\omega_{\FS}}
\newcommand{\mFS}{\mu_{\FS}}

Suppose, as in the previous subsection, that $n$ is an even natural number.

Let us identify $T^2$ with $T_1^2$ and consider the diffeomorphism
$$\phi:T^2\to T^2_n,\qquad\phi(([x],[y]))=([nx],[ny]).$$
Let $(x,y)\in\RR^2$ denote the canonical coordinates. These coordinates define
translation invariant vector fields $\partial_x,\partial_y$ on $\RR^2$, which
induce by projection vector fields on each $T^2_n$; we denote these vector fields on
$T_n^2$ with the same symbols $\partial_x,\partial_y$. We denote the dual forms on
$T^2_n$ by $dx,dy$.

\begin{lemma}
\label{lemma:curvature-A-n}
There exists a $\wh{\Gamma}_n$-invariant connection $A$ on $M_n\to T^2_n$
whose curvature $F_A$ satisfies
$$\phi^*F_A=2\pi\imag n\, dx\wedge dy.$$
\end{lemma}
\begin{proof}
Define a connection
$A$ on $M_n\to T_n^2$ by the prescription that its horizontal distribution
is $g_n$-orthogonal to the tangent spaces of the $S^1$-orbits.
Since the action of $\wh{\Gamma}_n$ on $M_n$ is by principal bundle automorphisms
and $g_n$-isometries, $A$ is $\wh{\Gamma}_n$-invariant.
To compute the curvature of $A$ we work on $T(\RR,\RR)$. Consider the matrices
$$m_x=\left(\begin{array}{ccc}
0 & 1 & 0 \\
0 & 0 & 0 \\
0 & 0 & 0 \end{array}\right),\qquad
m_y=\left(\begin{array}{ccc}
0 & 0 & 0 \\
0 & 0 & 1 \\
0 & 0 & 0 \end{array}\right)$$
and let $X,Y$ be the left invariant vector fields on $T(\RR,\RR)$ whose restrictions
to $T_{\Id}T(\RR,\RR)$ are given by $m_x,m_y$ respectively. The vector fields
$X,Y$ descend to $S^1$-invariant horizontal vector fields $X',Y'$ on $M_n$ whose
projections to $T^2_n$ satisfy $D\rho(X')=\partial_x$ and $D\rho(Y')=\partial_y$. On the other hand,
$[X,Y]$ is the left invariant vector field whose restriction to $T_{\Id}T(\RR,\RR)$
is equal to $[m_x,m_y]$. The latter can easily be identified with the restriction of
$2\pi n^{-1}\xX$ to $T_{\Id}T(\RR,\RR)$, where $\xX$ is the vector field on $T(\RR,\RR)$ induced
by the infinitesimal action of $\imag\in\Lie S^1$ that results from deriving the action (\ref{eq:accio-cercle-Mn}).
It follows that
$F_A=2\pi\imag n^{-1}\,dx\wedge dy$. Since $\phi^*dx=n\,dx$ and $\phi^*dy=n\,dy$, the result follows.
\end{proof}

Incidentally, note that Lemma \ref{lemma:curvature-A-n} implies
by Chern--Weil theory that $\deg M_n=n$ which, combined with
Lemmas \ref{lemma:abelian-subgroups-Gamma-n} and
\ref{lemma:abelian-subgroups-wh-Gamma}, implies that the first
(resp. second) statement of Proposition
\ref{prop:accions-en-fibracions} is sharp for line bundles $L$ of even
degree satisfying $|\deg L|\geq 8$
(resp. for any $L$).

Define
$$P_n=\phi^*M_n,\qquad A_n=\phi^*A,$$
so that $P_n$ is a principal circle bundle over $T^2$ carrying an effective action of
$\Gamma_n$ and $A_n$ is a $\Gamma_n$ invariant connection on $P_n$ whose curvature
is equal to $F_{A_n}=2\pi\imag n\,dx\wedge dy$.

Let us identify $S^2$ with the unit sphere centered at $0$ in
$\RR^3$, and consider the action of $S^1$ on $S^2$ given by
rotations around the $z$-axis:
\begin{equation}
\label{eq:accio-cercle-cp1}
e^{2\pi\imag t}\cdot(x,y,z)=(x\cos t-y\sin t,x\sin t+y\cos t,z).
\end{equation}
Let $\oFS$ be the volume form associated to restriction of the
Euclidean metric on $S^2$ and the orientation specified by the
ordered basis $(\partial_x,\partial_y)$ of $T_{(0,0,1)}S^2$. We
may look at $\oFS$ as a symplectic form on $S^2$, with respect
to which the action of $S^1$ given by rotation is Hamiltonian.
The moment map $\mFS:S^2\to\imag\RR$ is
$$\mFS(x,y,z)=\imag z,$$
so $\mFS(S^2)=\imag[-1,1]$.
We have
\begin{equation}
\label{eq:volum-ofs}
\int_{S^2}\oFS=4\pi.
\end{equation}
Consider the associated bundle $P_n\times_{S^1}S^2$ and the
projection
$$\Pi_n:P_n\times_{S^1}S^2\to T^2.$$
We are next going to construct a $\Gamma_n$-invariant
symplectic form on $P_n\times_{S^1}S^2$ using the minimal
coupling construction (see e.g. \cite[\S 6.1]{MS}). In order to
keep track of the cohomology class represented by the
symplectic form we will give the construction in some detail.

Let $D\Pi_n$ denote the vertical tangent bundle of the
fibration $\Pi_n$. Each fiber of $\Pi_n$ can be identified, in
a way unique up to the action of $S^1$, with $S^2$. Since
$\oFS$ is $S^1$-invariant it defines, via these
identifications, a section $\omega_0^{\vert}$ of
$\Lambda^2(\Ker D\Pi_n)^*$. On its turn, the connection $A_n$
induces a left inverse of the inclusion $\Ker
D\Pi_n\hookrightarrow T(P_n\times_{S^1}S^2)$ which when
combined with $\omega_0^{\vert}$ leads to a $2$-form
$$\wt{\omega}_0\in\Omega^2(P_n\times_{S^1}S^2)$$
whose restriction to each fiber coincides with $\oFS$.
The form $\wt{\omega}_0$ is not closed (unless $A_n$ is flat),
but the following $2$-form is closed (see e.g. \cite[Theorem
7.34]{BGV}):
\begin{equation}
\label{eq:def-omega-0}
\omega_0=\wt{\omega}_0+\mFS\cdot \Pi_n^*F_{A_n}.
\end{equation}

\begin{lemma}
\label{lemma:omega-delta} For any real number $\delta>2\pi n$
$$\omega_{\delta}=\omega_0+\delta \Pi_n^*(dx\wedge dy)$$ is a
$\wh{\Gamma}_n$-invariant symplectic form on
$P_n\times_{S^1}S^2$.
\end{lemma}
\begin{proof}
It is clear that $\omega_{\delta}$ is closed (this holds
regardless of the value of $\delta$), so we prove that
$\omega_{\delta}$ is nondegenerate if $\delta>2\pi n$.
The vertical and horizontal distributions in
$T(P_n\times_{S^1}S^2)$ are $\omega_{\delta}$-orthogonal so it
suffices to prove that the restrictions of $\omega_{\delta}$ to
both distributions are nondegenerate. The restriction to the
vertical distribution coincides with $\omega_0^{\vert}$, which
is nondegenerate  because it coincides on each fiber with
$\oFS$. To prove that the restriction to the horizontal
distribution is nondegenerate if $\delta>2\pi n$, use 
$F_{A_n}=2\pi\imag n\,dx\wedge dy$ and $|\mu(u)|\leq
1$ for every $u\in S^2$. Finally, to prove that
$\omega_{\delta}$ is $\wh{\Gamma}_n$-invariant observe that
$\omega_0$ is $\wh{\Gamma}_n$-invariant (this is a consequence
of the invariance of the connection $A_n$), and that $dx\wedge
dy$ is invariant under the action of $B_n$ (see the proof of
Lemma \ref{lemma:abelian-subgroups-wh-Gamma}) on $T^2$ given by
conjugating the action on $T^2_n$ via the diffeomorphism
$\phi:T^2\to T^2_n$.
\end{proof}

\subsection{Completion of the proof}
\label{ss:proof-thm:main-2}

The action (\ref{eq:accio-cercle-cp1}) factors through a
morphism $$S^1\to\SO(3,\RR)$$ (via the standard action of
$\SO(3,\RR)$ on $S^2$) which represents an element of order $2$
in $\pi_1(\SO(3,\RR))\simeq\ZZ_2$. Hence for every even natural
number $n$ there is a diffeomorphism $\psi_n:T^2\times S^2\to
P_n\times_{S^1}S^2$ satisfying $\Pi_n\circ\psi_n=\Pi$ (recall
that $\Pi:T^2\times S^2\to T^2$ is the projection).

\begin{lemma}
\label{lemma:cohomology-class} For any $n\in 2\NN$ we have
$[\psi_n^*\omega_{\delta}]=\delta[\omega_{T^2}]+4\pi[\omega_{S^2}].$
\end{lemma}
\begin{proof}
It suffices to prove that
$[\psi_n^*\omega_0]=4\pi[\omega_{S^2}]$. Let
$\sigma_0,\sigma_1\subset P_n\times_{S^1}S^2$ be the
submanifolds corresponding to the fixed points $(0,0,1)$,
$(0,0,-1)$ respectively of the action of $S^1$ on $S^2$, i.e.
$$\sigma_0=P_n\times_{S^1}\{(0,0,1)\},\qquad \sigma_1=P_n\times_{S^1}\{(0,0,-1)\},$$
and let $S_j=\psi_n^{-1}(\sigma_j)$. Orient $\sigma_j$ and $S_j$ so that their
projections to $T^2$, which are diffeomorphisms, are orientation preserving. Since
$S_1,S_2$ are disjoint, a simple computation using the intersection product on
$H_*(T^2\times S^2)$ proves that the homology
classes represented by $S_j$ are
$$[S_0]=[T^2]+k[S^2],\qquad [S_1]=[T^2]-k[S^2]$$
for some integer $k$. It follows that for any $s\in S^2$
$$\int_{T^2\times\{s\}}\psi_n^*\omega_0=
\frac{1}{2}\left(\int_{S_0}\psi_n^*\omega_0+\int_{S_1}\psi_n^*\omega_0\right)=
\frac{1}{2}\left(\int_{\sigma_0}\omega_0+\int_{\sigma_1}\omega_0\right).$$
Since $\mFS([1:0])+\mFS([0:1])=0$, it follows from the definition of $\omega_0$
(\ref{eq:def-omega-0}) that
$$\frac{1}{2}\left(\int_{\sigma_0}\omega_0+\int_{\sigma_1}\omega_0\right)=0.$$
Consequently $[\psi_n^*\omega_0]=\beta[\omega_{S^2}]$ for some
real number $\beta$. But $\beta$ coincides with the total
volume of $\oFS$, which by (\ref{eq:volum-ofs}) is equal to
$4\pi$.
\end{proof}

We are now ready to prove Theorem \ref{thm:main-2}. Let
$\omega$ be an arbitrary symplectic form on $T^2\times S^2$.
Let $\alpha=\alpha(\omega)$ and $\beta=\beta(\omega)$, let
$n=\lambda(\omega)$ and let $\xi=\alpha/\beta$. Suppose that
$n$ is an even natural number satisfying $n\geq 8$. It follows, combining
Lemma \ref{lemma:omega-delta} and Lemma
\ref{lemma:cohomology-class}, that there exists a
$\wh{\Gamma}_n$-invariant symplectic form $\omega_{4\pi\xi}$ on
$P_n\times_{S^1}S^2$ satisfying
$$\frac{\beta}{4\pi}[\psi_n^*\omega_{4\pi\xi}]=[\omega].$$
By Lalonde and McDuff's Theorem \ref{thm:LM} there is a diffeomorphism
$\phi$ of $T^2\times S^2$ such that
$$\frac{\beta}{4\pi}\psi_n^*\omega_{4\pi\xi}=\phi^*\omega.$$
Since two symplectic forms that differ by multiplication by a
constant have identical symplectomorphism groups, it follows
that there is a subgroup of $\Symp(T^2\times S^2,\omega)$ which
is isomorphic to $\wh{\Gamma}_n$. Applying Lemma
\ref{lemma:abelian-subgroups-wh-Gamma}, the proof of the first statement of Theorem
\ref{thm:main-2} is complete.
It only remains to prove that the action of $\Gamma$ on the cohomology of
$T^2\times S^2$ is trivial. This follows from the next lemma.

\begin{lemma}
\label{lemma:symp-hom-trivial}
For any symplectic form $\omega$ on $T^2\times S^2$ and any symplectomorphism
$\phi$ of $\omega$ the action of $\phi$ on $H^*(T^2\times S^2;\ZZ)$ is trivial.
\end{lemma}
\begin{proof}
It suffices to prove that the action of $\phi$ on $H_2(T^2\times S^2;\RR)$ is trivial.
Let $\kappa_{S^2},\kappa_{T^2}\in H_2(T^2\times S^2;\ZZ)$ be as before the classes represented
by $\{t\}\times S^2$ and $T^2\times\{s\}$ respectively, for any $t\in T^2$ and $s\in S^2$.
Since $\pi_2(T^2)$ is trivial, $(\Pi\circ\phi)_*\kappa_{S^2}=0$, so $\phi_*\kappa_{S^2}=\lambda\kappa_{S^2}$ for some $\lambda\in\ZZ$. Since $[\omega]=\alpha(\omega)[\omega_{T^2}]+\beta(\omega)[\omega_{S^2}]$ with $\beta(\omega)\neq 0$
and $\omega_{T^2}$ pairs trivially with $\kappa_{S^2}$, we have $\lambda=1$.
The proof is finished using Lemma \ref{lemma:O-1-1} and the arguments preceding it.
%
\end{proof}

\section{Proof of Theorem \ref{thm:main-p}}
\label{s:proof-thm:main-p} We first prove that if $\omega$ is a
symplectic form on $T^2\times S^2$, $p>3$ is a prime such that
$2p>\lambda(\omega)$, and $\Gamma\subset\Symp(T^2\times
S^2,\omega)$ is a finite $p$-group, then $\Gamma$ is abelian.
This follows from the same arguments as in the proof of Theorem
\ref{thm:main-1}. The difference with the general situation
considered in Theorem \ref{thm:main-1} is that when applying
Lemmas \ref{lemma:esfera} and \ref{lemma:tor} to a $p$-group
$H$ with $p>3$, the subgroup $H'$ whose existence is claimed
turns out to be $H$ itself in both lemmas. When we apply
Proposition \ref{prop:accions-en-fibracions} during the proof
of Theorem \ref{thm:main-1} there are three possible outcomes,
which in the context of a finite $p$-group $\Gamma$ (with $p>3$
and $2p>\lambda(\omega)$) simplify as follows. If
$\Gamma_S=\{1\}$ then the abelian subgroup $A\subseteq\Gamma$
which is constructed turns out to be $\Gamma$ itself, so
$\Gamma$ is abelian. In the two other cases, the group
$\Gamma_0$ coincides with $\Gamma$, and similarly $\Gamma_1$ is
also equal to $\Gamma$.
The proof that $\Gamma$ is abelian is completed by observing
that, in Proposition \ref{prop:accions-fibrats-linia}, if
$\Gamma$ is a $p$-group ($p>3$), $\deg L$ is even, and $2p>\deg
L$, then $\Gamma_{\ab}=\Gamma$. To justify this, first note
that it suffices to consider the second statement (again
because in Lemma \ref{lemma:tor} for a $p$-group $H$, $p>3$,
the subgroup $H'$ coincides with $H$). The fact that $\deg L$
is even and $2p>\deg L$ implies that $p$ does not divide $\deg
L$. This implies, using Lemma \ref{lemma:grau-divisible}, that
$|\Gamma_B|$ divides $d_c=|[\Gamma,\Gamma]|$. In particular
$|\Gamma_B|\leq |[\Gamma,\Gamma]|$. By (2) in Lemma
\ref{lemma:propietats-Q} this implies that $\Gamma_B$ is
cyclic, because the exponent of $[\Gamma,\Gamma]$ is not
greater than the exponent of $\Gamma_B$, and $[\Gamma,\Gamma]$
is cyclic. Then (1) in Lemma \ref{lemma:propietats-Q} tells us
that $\Gamma$ is abelian.

Now suppose that $p>3$ is prime and that
$2p\leq\lambda(\omega)$. By the arguments in the proof of
Theorem \ref{thm:main-2} (see Subsection
\ref{ss:proof-thm:main-2}) there is a subgroup of
$\Symp(T^2\times S^2,\omega)$ isomorphic to $\Gamma_{2p}$. The group $\Gamma_p$
is isomorphic to
$$\la X,Y,Z\mid X^p=Y^p=Z^p=[X,Z]=[Y,Z]=1,\,[X,Y]=Z\ra,$$
so it suffices to prove that $\Gamma_{2p}$ has a subgroup
isomorphic to $\Gamma_p$. The map
$$d:T(\ZZ,\ZZ)\to T(\ZZ,\ZZ),\quad d(A(x,y,z))=A(2x,2y,4z)$$
is an injective morphism of groups and
$d^{-1}(T(\ZZ,2p\ZZ))=T(\ZZ,p\ZZ)$. Hence, $d$ gives an
injection
$$\Gamma_p=T(\ZZ,\ZZ)/T(\ZZ,p\ZZ)\hookrightarrow
T(\ZZ,\ZZ)/T(\ZZ,2p\ZZ)=\Gamma_{2p}$$ (in fact, computing
cardinals it is clear that we can identify the image of this
map with a $p$-Sylow subgroup of $\Gamma_{2p}$).

\section{Proof of Corollary \ref{cor:main}}
\label{s:proof-cor:main}

 Let $(M,\omega)$ be a symplectic manifold diffeomorphic to an $S^2$-fibration
over a compact Riemann surface $\Sigma$. If $\chi(\Sigma)\neq 0$ then $\chi(M)\neq 0$,
so by the main result in \cite{M2} the diffeomorphism group of $M$ is Jordan.
A fortiori, so is $\Symp(M,\omega)$. The only case not covered by \cite{M2} is precisely when
$\Sigma=T^2$. In this case, $M$ is either the trivial fibration $T^2\times S^2$ or a twisted
fibration. In the first case Theorem \ref{thm:main-1} applies. In the second case, we can consider a degree $2$ unramified covering
$\mu:T^2\to T^2$ and take the pullback $\mu^*M\to T^2$ of the fibration $M\to T^2$.
There is a degree $2$ unramified covering
$\nu:\mu^*M\to M$. Then $\mu^*M\simeq T^2\times S^2$, so $\Symp(\mu^*M,\nu^*\omega)$ is
Jordan by Theorem \ref{thm:main-1}, and the arguments in \cite[\S 2.3]{M1} imply, using $\nu$,
that $\Symp(M,\omega)$ is also Jordan.

Suppose now that $(M,\omega)$ is a symplectic manifold with $M$ diffeomorphic to
the product of two Riemann surfaces of genuses $g$ and $h$.
If $\chi(M)\neq 0$ then \cite{M2} implies as before that $\Symp(M,\omega)$
is Jordan. Now suppose that $\chi(M)=0$. Then $1\in\{g,h\}$, so suppose that $g=1$.
If $h=0$ then $M\simeq T^2\times S^2$, so by Theorem \ref{thm:main-1}
$\Symp(M,\omega)$ is Jordan. Finally, if $h\geq 1$ then one may find
cohomology classes $\alpha_1,\dots,\alpha_4\in H^1(M;\ZZ)$ such that
$\alpha_1\cup\dots\cup\alpha_4\neq 0$, so by \cite{M1} the diffeomorphism group
of $M$ is Jordan. Consequently, $\Symp(M,\omega)$ is Jordan in this case as well.

\appendix

\section{Automorphisms of Heisenberg groups and geometry}
Here we interpret geometrically the automorphism
$h\in\Aut(T(\RR,\RR))$ of \S\ref{ss:Z6-symmetry} in terms of iterated integrals
and the monodromy of certain fibration over the moduli space of elliptic curves. 
(See \cite{Os} for a group theoretical approach to
$\Aut(T(\ZZ,\ZZ))$ and \cite{OP} for an approach based on 
two dimensional local fields.)

It is easy to prove that
any automorphism of $\Gamma_{\RR}=T(\RR,\RR)$ lifting the automorphism of
$\Gamma_{\RR}/[\Gamma_{\RR},\Gamma_{\RR}]\simeq\RR^2$
given by $(x,y)\mapsto (-y,x+y)$ has to coincide with $h$ up to adding
linear combinations $\alpha x+\beta y$ to the third term, and all such lifts
have order $6$.
Varying $\alpha$ and $\beta$ corresponds to the action on $\Aut(\Gamma_{\RR})$
of the inner automorphisms of $\Gamma_{\RR}$, which act trivially on
$\Gamma_{\RR}/[\Gamma_{\RR},\Gamma_{\RR}]$.
In particular, the automorphism
$h'$ of $T(\RR,\RR)$ defined as
$$h'(A(x,y,z))=A\left(-y,x+y,z-xy-\frac{y^2-y}{2}\right)$$
represents the same class in $\Out(\Gamma_{\RR})$ as $h$.
%
Clearly $h'$ preserves\footnote{We remark that in Subsection
\ref{ss:Z6-symmetry} we use $h$ instead of $h'$ because $h$
preserves $T(\ZZ,n\ZZ)$ for even $n$, whereas $h'$ does not
preserve $T(\ZZ,n\ZZ)$ if $n$ is even.}
$\Gamma=T(\ZZ,\ZZ)$ 
and since $\Gamma_{\RR}=\Gamma\otimes\RR$ as nilpotent groups, one can recover $h'$ from
its restriction to $\Gamma$. The latter
belongs to $\Aut^+(\Gamma)$, the group of automorphisms of $\Gamma$ acting trivially
on $Z(\Gamma)\simeq\ZZ$.

Let $\Int^+(\Gamma)\subset\Aut^+(\Gamma)$ be the inner automorphisms.
The outer automorphism group $\Out^+(\Gamma)=\Aut^+(\Gamma)/\Int^+(\Gamma)$ maps to $\SL(2,\ZZ)$ via its action on
$\Gamma/[\Gamma,\Gamma]$, and one can prove easily that this morphism
$\eta:\Out^+(\Gamma)\to \SL(2,\ZZ)$ is injective. To prove that $\eta$ is also surjective,
consider a principal $S^1$-bundle $p:M\to T^2$ of degree $1$. For any $x_0\in M$ we have
$\pi_1(M,x_0)\simeq\Gamma$. Let
$F
\in\SL(2,\ZZ)$ and let $\phi:T^2\to T^2$
be a 
diffeomorphism whose mapping class coincides with $F$. Since
$\det F=1$, $\phi$ acts trivially on $H^2(T^2)$ and hence
admits a lift $\psi:M\to M$ which is a principal bundle
automorphism. Then $\psi$ defines an element
$\psi_*\in\Out^+(\Gamma)$ which only depends on $F$ and which
is mapped to $F$ by $\eta$. Therefore $\eta$ is surjective.

The preceding construction allows us to interpret
the quadratic terms in $h$ and $h'$ in terms of
Chen's iterated integrals (see e.g. \cite{H}). Let $u,v$ denote the standard coordinates in $\RR^2$ and let $du,dv$ be the induced $1$-forms in $T^2=\RR^2/\ZZ^2$. Denote also
for simplicity by $du,dv$ the  pullbacks to $M$.
Let $\alpha\in\Omega^1(M,\imag\RR)$ be a connection form whose curvature $d\alpha$ is equal to
$-2\pi \imag du\wedge dv$. For any smooth loop $\gamma$ in $M$ define
$$x(\gamma)=\int_{\gamma}du,\quad
y(\gamma)=\int_{\gamma}dv,\quad
z(\gamma)=\frac{-\imag}{2\pi}\int_{\gamma}\alpha+\int_{\gamma}
du\,dv.$$
$z(\gamma)$ 
is a homotopy functional by \cite[Proposition 3.1]{H}, i.e. it
only depends on the homotopy class of $\gamma$. One proves
easily that
$$\pi_1(M,x_0)\ni\gamma\mapsto A(x(\gamma),y(\gamma),z(\gamma))\in\Gamma$$
is an isomorphism of groups. Choose a lift $f\in\Aut^+(\Gamma)$ of $\psi_*$. We have
$f(A(x,y,z))=A(x',y',z+g(x,y,z)),$
where $(x',y')=F(x,y)$. Our aim is to compute $g(x,y,z)$ up to linear terms.
Suppose that $x=x(\gamma)$, $y=y(\gamma)$ and $z=z(\gamma)$. We have
\begin{align*}
g(x,y,z)
&=\frac{-\imag}{2\pi}\int_{\gamma}(\psi^*\alpha-\alpha)+\int_{\gamma} (\psi^*du\,\psi^*dv-du\,dv).
\end{align*}
Take $\phi:T^2\to T^2$ to be the map induced by the linear
transformation $F:\RR^2\to\RR^2$. Then $\phi^*(du\wedge
dv)=du\wedge dv$ because $\det F=1$, so
$\psi^*d\alpha=d\alpha$. Hence $\psi^*\alpha-\alpha$ is a
closed $1$-form, so its contribution to $g(x,y,z)$ is a linear
term on $x,y$ (closed $1$-forms only {\it see}
$\Gamma/[\Gamma,\Gamma]=H_1(M)=H_1(T)$). It follows that the
integral involving $\psi^*du\,\psi^*dv-du\,dv$ is a homotopy
functional which coincides with $g$ up to linear terms. By
naturality we have $\int_{\gamma}
(\psi^*du\,\psi^*dv-du\,dv)=\int_{p\circ \gamma}
(\psi^*du\,\psi^*dv-du\,dv),$ and since this is a homotopy
functional it only depends on the homotopy class of $\gamma$.
In particular we may assume that $\gamma$ comes from a linear
map $\RR\ni t\mapsto (\lambda t,\mu t)\in\RR^2$. Then
$\lambda=\int_{\gamma}du=x$, $\mu=\int_{\gamma}dv=y$,
and if $F=\left(\begin{array}{cc}\alpha & \beta \\ \delta & \epsilon \end{array}\right)$
a straightforward computation gives
$$\int_{p\circ \gamma} (\psi^*du\,\psi^*dv-du\,dv)=
\frac{\alpha\delta x^2+(\alpha\epsilon+\beta\delta-1)xy+\beta\epsilon y^2}{2}.$$
Taking $\alpha=0$, $\beta=-1$, $\delta=1$ and $\epsilon=1$ we obtain the quadratic terms in $h$ and $h'$.

We close this appendix constructing a morphism of groups
$\xi:\SL(2,\ZZ)\to\Aut^+(\Gamma)$ which is a section of
$\Aut^+(\Gamma)\to\Out^+(\Gamma)\simeq\SL(2,\ZZ)$. This gives a
conceptual explanation of the existence of elements of
$\Aut(\Gamma)$ of order $6$ (such as $h'$). Let $\mM=\mM_{1,1}$
be the moduli orbifold/stack of elliptic curves over $\CC$, let
$p:\cC\to\mM$ be the universal curve (all bundles here are to
be understood in the orbifold/stack sense), let
$\sigma:\mM\to\cC$ be the section corresponding to the marked
point, let $\dD=\sigma(\mM)$, let $\lL=\oO(\dD)\to\cC$, and let
$\lambda\in H^0(\lL)$ be a section transverse to the zero
section and satisfying $\lambda^{-1}(0)=\dD$. Let
$\tT=\sigma^*T^{\vert}\cC\to\mM$, where $T^{\vert}\cC$ is the
vertical tangent bundle of $\cC$. Let $\Lambda=\lL\otimes
p^*\tT^*$ and let $\Lambda^*\subset\Lambda$ be the
complementary of the zero section. Since $\lambda$ vanishes
transversely along $\dD$, its derivative defines a nonvanishing
section of $\Hom(\tT,\sigma^*\lL)$, which can be interpreted as
a section $b:\mM\to\Lambda^*$ lifting $\sigma$. Let
$r:\Lambda^*\to\mM$ be the projection and let $e_0\in\mM$ be
any point. We have $\pi_1(r^{-1}(e_0),b(e_0))\simeq\Gamma$,
and, thanks to the existence of the section $b$, the monodromy
defines a map $\SL(2,\ZZ)\simeq
\pi_1^{\orb}(\mM)\to\Aut^+(r^{-1}(e_0),b(e_0))\simeq\Aut^+(\Gamma)$
which is the desired morphism $\xi$.

\end{document}